\documentclass{article}

\usepackage[english]{babel}

\usepackage[letterpaper,top=2cm,bottom=2cm,left=3cm,right=3cm,marginparwidth=1.75cm]{geometry}

\usepackage{amsmath}
\usepackage{graphicx}
\usepackage[colorlinks=true, allcolors=blue]{hyperref}
\usepackage{amsmath, amssymb, amsthm, amsfonts}
\usepackage{fullpage}
\usepackage{hyperref}
\usepackage{enumitem}
\usepackage{mathtools}
\usepackage[capitalise]{cleveref}

\theoremstyle{definition}
\newtheorem*{definition}{Definition}
\theoremstyle{theorem}

\newtheorem{thm}{Theorem}[section]
\newtheorem{prop}[thm]{Proposition}
\newtheorem{lemma}[thm]{Lemma}
\newtheorem{conj}[thm]{Conjecture}

\theoremstyle{remark}

\newtheorem{claim}{Claim}

	\newcommand{\EE}{{\mathbb E}}

\newcommand{\cF}{\mathcal{F}}
\newcommand{\mpair}{m_{\textnormal{pair}}}

\newcommand{\eps}{\varepsilon}
\newcommand{\red}{\textnormal{red}}
\newcommand{\blue}{\textnormal{blue}}

\title{On off-diagonal hypergraph Ramsey numbers}
\author{David Conlon\thanks{Department of Mathematics, California Institute of Technology, Pasadena, CA 91125. Email: dconlon@caltech.edu. Research supported by NSF Award DMS-2054452.} \and
Jacob Fox\thanks{Department of Mathematics, Stanford University, Stanford, CA 94305. Email: jacobfox@stanford.edu. Research supported by NSF Award DMS-2154129.} \and
Benjamin Gunby\thanks{Department of Mathematics, Rutgers University, Piscataway, NJ 08854. Email: bg570@rutgers.edu.}\and
Xiaoyu He\thanks{Department of Mathematics, Princeton University, Princeton, NJ 08544. Email: xiaoyuh@princeton.edu. Research supported by NSF Award DMS-2103154.} \and
Dhruv Mubayi\thanks{Department of Mathematics, Statistics and Computer Science, University of Illinois, Chicago, IL 60607. Email: mubayi@uic.edu. Research partially supported by NSF Awards DMS-1763317,
DMS-1952767 and DMS-2153576, by a Humboldt Research Award and by a Simons Fellowship.} \and
Andrew Suk\thanks{Department of Mathematics, University of California at San Diego, La Jolla, CA 92093. Email: asuk@ucsd.edu. Research supported by an NSF
CAREER Award and by NSF Awards DMS-1952786 and DMS-2246847.} \and 
Jacques Verstra\"ete\thanks{Department of Mathematics, University of California at San Diego, La Jolla, CA 92093. Email: jacques@ucsd.edu. Research supported by NSF Award DMS-1800332.}}
\date{}

\begin{document}

\maketitle

\begin{abstract}
    A fundamental problem in  Ramsey theory is to determine the growth rate in terms of $n$ of the Ramsey number $r(H, K_n^{(3)})$ of a fixed $3$-uniform hypergraph $H$ versus the complete $3$-uniform hypergraph with $n$ vertices. We study this problem, proving two main results. First, we show that for a broad class of $H$, including links of odd cycles and tight cycles of length not divisible by three, $r(H, K_n^{(3)}) \ge 2^{\Omega_H(n \log n)}$. This significantly generalizes and simplifies an earlier construction of Fox and He which handled the case of links of odd cycles 
    and is sharp both in this case and for all but finitely many tight cycles of length not divisible by three. 
    Second, disproving a folklore conjecture in the area, we show that there exists a linear hypergraph $H$ for which $r(H, K_n^{(3)})$ is superpolynomial in $n$. This provides the first example of a separation between $r(H,K_n^{(3)})$ and $r(H,K_{n,n,n}^{(3)})$, since the latter is known to be polynomial in $n$ when $H$ is linear.
\end{abstract}

\section{Introduction}

Given two $k$-uniform hypergraphs (or, henceforth, $k$-graphs) $G$ and $H$, their \emph{Ramsey number} $r(G,H)$ is the smallest $N$ such that whenever the edges of the complete $k$-graph $K_N^{(k)}$ are colored in red or blue, there must be either a red copy of $G$ or a blue copy of $H$. That these numbers exist is the statement of Ramsey's original theorem~\cite{Ram}, but many questions remain about their quantitative behavior.

The graph case $k = 2$ has received particular attention. A result of Erd\H{o}s and Szekeres~\cite{ESz} from 1935 says that
\[r(K_s, K_n) \leq \binom{n + s - 2}{s-1}\]
and one of the main driving forces in the area has been to decide if and when this bound is close to being tight. When $s = n$, their bound implies that $r(K_n,K_n) \leq 4^n$ 
and the problem of improving this bound by an exponential factor, which was recently resolved in a breakthrough paper of Campos, Griffiths, Morris and Sahasrabudhe~\cite{CGMS}, has occupied a central place in extremal combinatorics. 

The problem that we will be concerned with in this paper is more closely related to the estimation of the off-diagonal Ramsey numbers $r(K_s, K_n)$ where $s$ is fixed and $n$ tends to infinity. In this case, the Erd\H{o}s--Szekeres bound $r(K_s, K_n) \leq n^{s-1}$ can be close to tight. For instance, by combining important results of Ajtai, Koml\'os and Szemer\'edi~\cite{AKS} and Kim~\cite{Ki}, we have
\[r(K_3,K_n)=\Theta\left(\frac{n^2}{\log n}\right).\]
Until lately, no similar result was known for any other $s$, but another recent breakthrough result by Mattheus and Verstra\"{e}te~\cite{MaV} shows that $r(K_4,K_n)= \tilde{\Theta}(n^3)$ (where the tilde in the big-O notation means that the estimate is tight up to logarithmic factors). At present, no similar result is known for $s \ge 5$.

Despite the problems that remain in the graph case, even less is known about Ramsey numbers of hypergraphs. It has been known since the 1970s that there are positive constants $c$ and $C$ depending only on $k$ such that
\begin{equation*}\label{powertower}
t_{k-1}(c n^2)\leq r(K_n^{(k)},K_n^{(k)})\leq t_{k}(C n),
\end{equation*}
where the tower function is given by $t_1(x) = x$ and $t_{i+1}(x) = 2^{t_i(x)}$. However, the correct tower height in these bounds remains unknown for all $k \geq 3$. The case $k = 3$ is of particular importance, since the ingenious stepping-up lemma of Erd\H{o}s and Hajnal (see, for instance,~\cite{GRS}) allows us, starting from $k = 3$, to construct lower bound colorings for uniformity $k+1$ from colorings for uniformity $k$, gaining an extra exponential each time. In particular, if we could show that $r(K_n^{(3)}, K_n^{(3)})$ grows as a double exponential in $n$, this would essentially resolve the problem of estimating $r(K_n^{(k)},K_n^{(k)})$ for all $k$.

Our concern in this paper will be with studying off-diagonal Ramsey numbers for $3$-graphs. As above, variants of the stepping-up lemma (see, for example,~\cite{CFS3, MSBLMS}) can be used to lift bounds on off-diagonal Ramsey numbers for $3$-graphs to higher uniformities, so we will focus on the critical $3$-uniform case. For $r(K_s^{(3)}, K_n^{(3)})$ with $s$ fixed and $n$ growing, the best known bounds, due to Conlon, Fox and Sudakov~\cite{CFS}, are that
\[2^{c n \log n} \leq r(K_s^{(3)}, K_n^{(3)}) \leq 2^{C n^{s-2} \log n}\]
for some positive constants $c$ and $C$ depending only on $s$. 

In general, there are very few $H$ for which the growth rate of $\log(r(H,K_n^{(3)}))$ in terms of $n$ is well understood. The only exceptions are tripartite $3$-graphs and their iterated blowups, for which Erd\H os and Hajnal~\cite{EH} proved  that $r(H,K_n^{(3)}) \le n^{O_H(1)}$, and links of odd cycles, for which Fox and He~\cite{FH} showed that
\[r(H,K_n^{(3)})=2^{\Theta_H(n\log n)},\]
proving that an old upper bound of Erd\H{o}s and Hajnal \cite{EH} is tight. Note that for us the \emph{link} of a ($2$-)graph $G$ will mean the $3$-graph $L_G$ with vertex set $V(G)\cup\{u\}$, where $u$ is a new vertex, and edge set $e\cup\{u\}$ for every $e\in E(G)$. More generally, the result of Fox and He showed that $r(H,K_{n,n,n}^{(3)})\geq 2^{\Omega_H(n\log n)}$ whenever $H$ is the link of an odd cycle. 

Our first theorem extends this latter result by showing that $r(H,K_{n,n,n}^{(3)})\geq 2^{\Omega_H(n\log n)}$ for a large class of $3$-graphs $H$, including links of odd cycles and all tight cycles of length $n$ with $3\nmid n$ and, as we shall see below, is sharp in both these cases. We state our result in terms of a function of $H$ that we denote by $\mpair(H)$, though we will hold off on formally defining this parameter until Section~\ref{sec:proofmpair}. 

\begin{thm}\label{thm:mainmpair}
    If $\mpair(H)\geq\frac{1}{2}$, then $r(H, K_{n,n,n}^{(3)}) \ge 2^{\Omega_H(n \log n)}$.
\end{thm}

Since $K_{n,n,n}^{(3)}\subseteq K_{3n}^{(3)}$, this theorem also holds with $K_{n,n,n}^{(3)}$ replaced by $K_n^{(3)}$. Though we will only define $\mpair(H)$ later on, some sense of it may be gained by noting that an early result of Erd\H{o}s and Hajnal \cite{EH} on off-diagonal hypergraph Ramsey numbers may also be couched in this language.

\begin{thm}\label{thm:EHmpair}[Rephrasing of a result in \cite{EH}]
    If $\mpair(H)>\frac{1}{3}$, then $r(H, K_{n,n,n}^{(3)}) \ge 2^{\Omega_H(n)}$.
\end{thm}

As well as being much more general than the result of Fox and He~\cite{FH}, Theorem~\ref{thm:mainmpair} also has a simpler proof, in particular avoiding the entropy techniques of~\cite{FH} and perhaps suggesting that $\mpair(H)$ is in fact the correct parameter to work with. It is also often tight, as shown by the following result. The \emph{shadow} (or \emph{$1$-skeleton}) of a $3$-graph $H$ is the graph $G$ with $V(G)=V(H)$ whose edge set consists of all pairs of vertices $uv$ such that  $uvw$ is an edge of $H$ for some vertex $w$. We use the standard notation $\partial H = E(G)$.

\begin{thm}\label{thm:nlognsupersaturation}
   Let $H$ be a $3$-graph that can be created by starting with the empty graph and iteratively adding edges that strictly increase the number of edges in the shadow of $H$ at each step. Then
   \[r(H, K_{n,n,n}^{(3)}) \le 2^{O_H(n \log n)}.\]
\end{thm}
Note that \cref{thm:nlognsupersaturation} immediately implies that \cref{thm:mainmpair} is tight for links of odd cycles and for tight cycles, as these may clearly be constructed by repeatedly adding $3$-edges that add at least one $2$-edge to the shadow. However, we also note that, apart from finitely many exceptions in the case of tight cycles, the stronger bound $r(H, K_{n}^{(3)}) \le 2^{O_H(n \log n)}$ was already known in both of these cases~\cite{EH, Mu}. 

Our second main result addresses a longstanding problem in the area, asking whether $r(F, K_n^{(3)})$ is polynomial in $n$ if $F$ is a \emph{linear} $3$-graph (also known as a linear triple system, partial triple system or partial Steiner system), that is, a $3$-graph where any two edges share at most one vertex. That this should be the case was something of a folklore conjecture, but had been resistant to attack. The following result shows why. 

\begin{thm}\label{thm:linearsteppingup}
For all sufficiently large $k$, there exists a linear $3$-graph $F$ on $k$ vertices such that $$r(F, K_{n}^{(3)}) > 2^{c(\log n)^{\sqrt{k}/(128\log^5k)}},$$ where $c > 0$ depends only on $k$.
\end{thm}

 In contrast, a result of Fox and He~\cite{FH} shows that if $F$ is a linear $3$-graph, then the Ramsey number $r(F, K_{n,n,n}^{(3)})$ grows polynomially in $n$. Combined, these results give the first examples of $3$-graphs $F$ for which the Ramsey numbers $r_3(F, K_{n,n,n}^{(3)})$ and $r_3(F, K_n^{(3)})$ are known to have very different growth rates.  

Theorem~\ref{thm:linearsteppingup} also has a further corollary. If, for a triple system $F$, we define $m_k(F)$ to be the minimum number of edges in an $F$-free triple system with chromatic number at least $k$, Bohman, Frieze and Mubayi~\cite{BFM} conjectured that there exists a linear triple system $F$ for which $m_k(F) = k^{3+o(1)}$.  It is easy to see that $m_k(F) \ge k^{3+o(1)}$, since any triple system with $f$ edges has chromatic number $O(f^{1/3})$, so this conjecture would be asymptotically tight. Theorem~\ref{thm:linearsteppingup} settles the conjecture, since it produces $N$-vertex $F$-free 3-graphs with independence number $n=N^{o(1)}$ and, hence, chromatic number at least $k = N^{1-o(1)}$. Moreover, the number of edges is trivially at most ${N \choose 3}=k^{3+o(1)}$.

We now proceed to our proofs, beginning with \cref{thm:mainmpair}. We will return to \cref{thm:linearsteppingup} in Section~\ref{sec:steppingupintro} and conclude in Section~\ref{sec:conclusion} with some further remarks and conjectures. We note that, unless otherwise indicated, all of our logarithms are natural logarithms.

\section{Proof of \cref{thm:mainmpair}}\label{sec:proofmpair}

In this section, we prove \cref{thm:mainmpair}, saying that if $\mpair(H) \geq \frac 12$, then $r(H, K_{n,n,n}^{(3)}) \geq 2^{\Omega_H(n \log n)}$. Our first task is to formally define $\mpair(H)$, which we do in Section~\ref{sec:avoid}. Our construction is described in detail in Section~\ref{sec:pair} and then, in Section~\ref{sec:endgame}, we show that it has the required properties. In Section~\ref{sec:mpaircomputation}, we show that $\mpair(H) \ge \frac 12$ for all links of odd cycles and all tight cycles whose length is not a multiple of $3$, so that our results do indeed apply in these cases. Finally, in Section~\ref{sec:supersat}, we give the short proof of \cref{thm:nlognsupersaturation}, which says that \cref{thm:mainmpair} is tight for a large class of $3$-graphs.

\subsection{Avoidability}\label{sec:avoid}

In this section, we formally define the ``pair density'' statistic $\mpair(H)$ that was mentioned in the introduction and give a characterization of all $3$-graphs $H$ with $\mpair(H) \ge \frac{1}{2}$. We will call such $3$-graphs ``avoidable''. However, before getting to that, we need several other definitions.

\begin{definition}
A hypergraph is \emph{ordered} if its vertex set is a linearly ordered set and \emph{nonempty} if it has at least one edge.
\end{definition}

\begin{definition}
    An \textit{oriented $3$-graph} $G$ is a pair of sets $(V,E)$ such that $E\subseteq V^3$ contains at most one of the six permutations of any given ordered triple $(v_1, v_2, v_3) \in V^3$.
\end{definition}

\begin{definition}
    Given an ordered $3$-graph $H$ and an oriented $3$-graph $G$, a \textit{pair homomorphism} from $H$ to $G$ is a function $f:\partial H\to V(G)$ such that, for every edge $\{u,v,w\} \in E(H)$ with $u<v<w$,  $f(uv)$, $f(vw)$ and $f(wu)$ are distinct and $(f(uv),f(vw),f(wu)) \in E(G)$. The map sending $uvw$ to $(f(uv),f(vw),f(wu))$ induces another map, which we also call $f$, from $E(H)$ to  $E(G)$. Write $f(H)$ for the subgraph of $G$ whose edge set is the image of this map and whose vertex set is $f(\partial H)$. Note that $v(f(H)) = |f(\partial H)|$ and
    $e(f(H))=|f(E(H))|$.
    \end{definition}

\begin{definition}
    If $H$ is a nonempty  ordered $3$-graph, define
    \[
    \mpair(H)\coloneqq \min_{f:\partial H\to V(G)} \max_{H'\subseteq H}\frac{e(f(H'))}
       {v(f(H'))},
    \]
   where $f$ ranges over all pair homomorphisms  from $H$ to any oriented $G$ and $H'$ ranges over all nonempty subhypergraphs of $H$. We define $\mpair(H)$ for an unordered $3$-graph $H$ to be the minimum of $\mpair(H)$ over all orderings of $H$.
\end{definition}

Since $f(H)$ has no isolated vertices by definition, $v(f(H)) \le \sum \hbox{deg}_{f(H)}(v) = 3e(f(H))$ and consequently $\mpair(H) \ge 1/3$ for all nonempty $H$. On the other hand, $\mpair(K_s^{(3)}) = \Theta (s)$. Indeed,  the upper bound follows immediately from the Kruskal--Katona theorem and for the lower bound we may argue as follows. Suppose that $f$ is a surjective pair homomorphism from $K_s^{(3)}$ to an oriented graph $G$. Then, for any incident pair $e$ and $e'$ in $\partial K_s^{(3)}$, we know that $f(e) \ne f(e')$ and $e$ and $e'$ lie in an edge of $G$. As we vary $e'$ over all elements of $\partial K_s^{(3)}$ incident to $e$, we see that $f(e)$ lies in at least $\Omega(s)$ edges of $G$. Hence, the minimum degree of $G$ is at least $\Omega(s)$, so $e(G)/v(G) = \Omega(s)$ as well.

In order to say what it means for a $3$-graph to be avoidable, we must first recall that a Berge cycle is a hypergraph with (distinct) vertices $v_1, \ldots, v_t$ and distinct edges $e_1, \ldots, e_t$, where $\{v_i, v_{i+1}\} \subseteq e_i$ and indices are taken modulo $t$.

\begin{definition}
We say that a $3$-graph $H$ is \textit{avoidable} if every $3$-graph $G$ for which there is a pair homomorphism $f$ from $H$ to $G$  contains a Berge cycle.
\end{definition}

The following lemma shows that being avoidable is equivalent to having pair density at least $1/2$.

\begin{lemma}\label{lem:avoidmpair}
    A $3$-graph $H$ is avoidable if and only if $\mpair(H) \ge \frac{1}{2}$.
\end{lemma}

\begin{proof}
    This boils down to the observation that a $3$-graph $G$ contains a subhypergraph $G'$ with $e(G')/v(G')\ge \frac{1}{2}$ if and only if it contains a Berge cycle. Indeed, any Berge cycle satisfies this property and, conversely, any Berge acyclic hypergraph satisfies $v(G) > 2e(G)$ by applying \cref{treelemma} below to each connected component.
\end{proof}

The following definition is also of interest.

\begin{definition}
    A $3$-graph $H$ is $123$-inducible if there is an ordering of $V(H)$ and a labelling of $\partial H$ by  $1$, $2$ and $3$ in such a way that, for every  $uvw \in E(H)$ with $u<v<w$, $uv$ is labelled $1$, $vw$ is labelled $2$ and $uw$ is labelled $3$. 
\end{definition}

\begin{lemma}
    $H$ is $123$-inducible if and only if $\mpair(H)=\frac{1}{3}$.
\end{lemma}

\begin{proof}
    As a $123$-coloring is a pair homomorphism yielding one edge and three vertices (and, clearly, $3e \geq v$ for any $3$-graph without isolated vertices), any $123$-inducible graph has $\mpair$ equal to $1/3$. Conversely, if $\mpair(H)=\frac{1}{3}$, there is a pair homomorphism $f$ such that $3e(f(H))=v(f(H))$. But then $f(H)$ must consist entirely of disjoint triples, whose vertices we may color with the colors $1,2,3$ in accordance with the ordering on $v(H)$. That is, $H$ is $123$-inducible.
\end{proof}

We note that Theorem~\ref{thm:EHmpair}, the result of Erd\H{o}s and Hajnal~\cite{EH} quoted in the introduction, was originally proved in terms of $123$-inducibility and then states that if a $3$-graph $H$ is not $123$-inducible, then $r(H, K_{n,n,n}^{(3)}) \geq 2^{\Omega_H(n)}$. The proof is rather simple: take a random coloring of the edges of $K_N^{(2)}$ in three colors $1$, $2$ and $3$ and then color an edge $uvw$ of $K_N^{(3)}$ with $u < v < w$ in red if and only if $uv$ is colored $1$, $vw$ is colored $2$ and $uw$ is colored $3$. It is easy to verify that there exists a positive constant $C$ such that, with positive probability, this coloring contains no red copy of any $3$-graph $H$ which is not $123$-inducible and no blue $K_{n,n,n}^{(3)}$ with $n = C \log N$.

\subsection{The construction}\label{sec:pair}

In this section, we describe our construction for \cref{thm:mainmpair}. In light of \cref{lem:avoidmpair}, it will suffice to show the following.

\begin{thm}\label{thm:main}
    If $H$ is avoidable, then $r(H, K_{n,n,n}^{(3)}) \ge 2^{\Omega_H(n \log n)}$.
\end{thm}

\vspace{3mm}

The construction for \cref{thm:main} is as follows. Take $\eps$ sufficiently small in terms of $H$, $c$ sufficiently small in terms of $\eps$ and $n$ sufficiently large in terms of $c$. Let $N = 2^{c n\log n}$ and $t = n^{\eps}$ and define $f:\binom{[N]}{2}\to [t]$ uniformly at random. Pick $g:[t]^3 \to \{\red, \blue\}$ to satisfy:

\begin{enumerate}
    \item The red color in $g$ is an oriented $3$-graph, i.e., for all $e\in [t]^3$ at most one of the six permutations $\sigma(e)$ of $e$ satisfies $g(\sigma(e)) = \red$. In particular, if $e$ has at least two equal coordinates, then $g(e) = \blue$.
    \item Among any $k\le {v(H)\choose 2}$ elements of $[t]$ there are  fewer than $k/2$ red edges. In particular, the red edges in $g$ form a linear $3$-graph.
    \item In blue, $g$ does not have complete balanced oriented  tripartite subgraphs of order $t^{1-\eps}$. That is, for any $X, Y, Z \subseteq [t]$ of size $t^{1-\eps}$, there must be at least one triple $(x,y,z)\in X\times Y \times Z$ for which $g(x,y,z) = \red$.
\end{enumerate}

Thus, the red color in $g$ may be seen as an oriented $3$-uniform analog of a graph with high girth and small independence number. Our coloring $\chi:\binom{N}{3}\to \{\red, \blue\}$ is now defined by $\chi(u,v,w) = g(f(uv),f(vw), f(wu))$ for all $u<v<w$ in $[N]$. Note that condition 2 already implies that $\chi$ contains no red copy of $H$ provided $H$ is avoidable. 

We now prove the existence of a function $g$ satisfying conditions 1-3 before moving on to proving Theorem~\ref{thm:main} in the next section.

\begin{lemma}
    For any $3$-graph $H$, there exists a function $g$ satisfying conditions 1-3 above for any $\varepsilon$ sufficiently small and $t$ sufficiently large in terms of $H$.
\end{lemma}

\begin{proof}
    We would like to construct a function $g$ with the required properties. First, note that we must have $g(i,i,i) = g(i,i,j)=g(i,j,i)=g(j,i,i) = \blue$ for all $i,j$. 
    To color the remaining edges, we use the alteration method. Let $g:[2t]^3\to \{\red,\blue\}$ be random, where $g(i,j,k)=\red$ with probability $p = t^{-2+\delta}$ independently for each triple $(i,j,k)$. The expected number of pairs of permutations of the same triple which are both colored red is at most
    \[\binom{2t}{3} \binom{6}{2} p^2 = O(t^3 p^2) < t/3\]
    and the expected number of red $(k,k/2)$-configurations is at most 
    \[
    \binom{2t}{k}\binom{k^3}{k/2}p^{k/2} = O_k(t^k p^{k/2}) < t/3
    \]
    if $\delta$ is sufficiently small and $t$ is sufficiently large in terms of $|V(H)|$. The expected number of blue tripartite $X, Y, Z\subset [2t]$ of size $|X|=|Y|=|Z|=t^{1-\eps}$ is at most
    \[
    (2t)^{3t^{1-\eps}}(1-p)^{t^{3-3\eps}-O(t^2)} \le e^{O(t^{1-\eps}\log t) - t^{1-3\eps+\delta}} < t/3
    \]
    if $\eps$ is sufficiently small compared to $\delta$. Here the $O(t^2)$ subtracted in the exponent accounts for the possibility that $X$, $Y$, $Z$ overlap, in which case up to $O(t^2)$ triples $(x,y,z)$ with a repeated coordinate were colored blue in advance. 
    
    Thus, in total, there exists a sample of $g$ for which the number of bad red and blue configurations is together less than $t$, so we may delete $t$ vertices from $[2t]$ to obtain the desired $g$.
\end{proof}

\subsection{Proof of \cref{thm:main}} \label{sec:endgame}
We note that $g$ is fixed and all probabilistic statements are with respect to the random choice of $f$. By construction, for all choices of $f$, the coloring $\chi$ is guaranteed to contain no red copy of $H$. It therefore remains to show that with positive probability $\chi$ contains no blue copy of $K_{n,n,n}^{(3)}$.  We remark that the only property of $g$ that we will use in the proof of this statement is property 3, which will be used in the proof of Claim 2 below. 

Observe that every $K_{n,n,n}^{(3)}$ contains a complete tripartite subgraph with parts $I$, $J$, $K$ of order $n/3$ such that $I < J < K$ (that is, for all $i\in I, j\in J, k\in K$, $i<j<k$). Thus, up to this factor of three that we will ignore, it suffices to show that $\chi$ contains no blue copy of $K_{n,n,n}^{(3)}$ whose parts $I, J, K$ satisfy $I < J < K$. By taking a union bound over all $N^{3n}$ possible choices of $I, J, K$, we see that it suffices to show the following.

\begin{lemma}\label{lem:blue}
    If $I < J < K$ are three $n$-element subsets of $[N]$, then the probability that $\chi(i,j,k)=\blue$ for all $(i,j,k)\in I\times J \times K$ is at most $t^{-\eps n^2/50}$.
\end{lemma}

\begin{proof}
We would like to count \textit{blue functions} $f:(I\times J) \cup (J\times K) \cup (I \times K) \to [t]$, i.e., those for which $g(f(i,j), f(j,k), f(k,i)) = \blue$ for all $(i,j,k)\in I\times J \times K$. Let $\cF$ be the family of all such blue $f$. If we set $\delta = \eps/50$, \cref{lem:blue} is equivalent to the inequality 
\begin{equation}\label{eq:goal}
|\cF| \le t^{(3-\delta)n^2},
\end{equation}
which is now our goal. We will split the argument into three steps. For simplicity, we often abuse notation by writing $\sum_X f(x)$ and $\prod_X f(x)$ instead of  $\sum_{x \in X} f(x)$ and $\prod_{x \in X} f(x)$.

\vspace{3mm}

\noindent {\bf 1. Coloring $(I\times J)\cup(J\times K)$ first.} Let $\cF_J$ be the family of functions $f_J: (I\times J)\cup(J\times K) \to [t]$ for which there are at least $t^{(1-2\delta)n^2}$ blue functions $f\in \cF$ extending $f_J$, i.e., for which $f|_{(I\times J)\cup(J\times K)} = f_J$. It suffices to show that 
\begin{equation}\label{eq:goal_J}
|\cF_J| \le t^{(2-2\delta)n^2},
\end{equation}
since that would imply that
\[
|\cF| \le |\cF_J| \cdot t^{n^2} + t^{2n^2} \cdot t^{(1-2\delta)n^2} \le 2 t^{(3-2\delta )n^2} \le t^{(3-\delta)n^2},
\]
by conditioning on whether or not $f|_{(I\times J)\cup(J\times K)} \in \cF_J$. That is, \eqref{eq:goal} would hold.

It will be convenient to set up some notation for counting blue extensions of $f_J$. For $f_J: (I\times J)\cup(J\times K) \to [t]$, let $f_J^*(i,k)$ be the number of choices of $c = f(i,k)$ for which $g(f_J (i,j), f_J(j,k), c) = \blue$ for all $j$. Observe that the number of extensions $P(f_J)$ of $f_J$ to blue functions $f\in \cF$ is exactly the product $P(f_J) = \prod_{I\times K}f_J^*(i,k)$. Thus, $\cF_J$ is exactly the family of $f_J$ satisfying $P(f_J) \ge t^{(1-2\delta)n^2}$, so, by the AM-GM inequality, every $f_J \in \cF_J$ satisfies
\[
S(f_J) \coloneqq \sum_{I\times K} f_J^*(i,k) \ge n^2 \left(\prod_{I\times K}f_J^*(i,k)\right)^{1/n^2} = n^2 P(f_J)^{1/n^2} \ge n^2 t^{1-2\delta}.
\]

\vspace{3mm}

\noindent {\bf 2. Coloring $(I\times \{j\})\cup(\{j\}\times K)$ one $j$ at a time.} To prove \eqref{eq:goal_J}, we reveal the values of $f_J$ one $j$ at a time and track the evolution of $P(f_J)$ and $S(f_J)$. To do so, we extend the definitions of $f^*_J$, $P(f_J)$ and $S(f_J)$ to subsets $J'\subseteq J$. Given $J'\subseteq J$ and a function $f_{J'}: (I\times J')\cup(J'\times K)\to [t]$ defined on the pairs incident to $J'$, we take $f_{J'}^*(i,k)$ to be the number of choices of $c=f(i,k)$ for which $g(f_{J'}(i,j), f_{J'}(j,k),c) = \blue$ for all $j\in J'$, $P(f_{J'}) \coloneqq \prod_{I\times K}f_{J'}^*(i,k)$ and $S(f_{J'}) = \sum_{I\times K} f^*_{J'}(i,k)$.

If $J = \{j_1,\ldots, j_n\}$, let $J_r = \{j_1,\ldots, j_r\}$, where $J_0 = \emptyset$. Observe that if $f_J: (I\times J)\cup(J\times K)\to [t]$ is restricted to $f_{J_r} \coloneqq f_J |_{(I\times J_r)\cup(J_r\times K)}$, then the values $S(f_{J_r})$ satisfy
\[
n^2 t = S(f_{J_0}) \ge S(f_{J_1}) \ge S(f_{J_2}) \ge \cdots \ge S(f_{J_n}) = S(f_J) \ge 0.
\]
Thus, for any $f_J$ we have that at least $n/2$ values of $r$ satisfy $S(f_{J_{r-1}}) - S(f_{J_r}) \le 2nt$. In words, this means that revealing the colors on the pairs incident to $j_r$ decreases the total number of blue choices between $I$ and $K$ by at most $2nt$. Call $r$ \textit{slow} (with respect to $f_J$) if $S(f_{J_{r-1}}) - S(f_{J_r}) \le 2nt$. It will suffice to show the following.

\begin{claim}\label{claim:slow}
If $P(f_{J_{r-1}}) \ge t^{(1-2\delta)n^2}$, then the number of ways to color $f_{j_r}\coloneqq f_J |_{(I\times \{j_r\})\cup(\{j_r\}\times K)}$ to make $r$ slow is at most $t^{(2-5\delta)n}$.
\end{claim}

Indeed, if the claim is true, we have
\[
|\mathcal{F}_J| \le 2^n \cdot (t^{2n})^{n/2} \cdot (t^{(2-5\delta)n})^{n/2} \le t^{(2-2\delta)n^2},
\]
since, to construct $f_J\in \mathcal{F}_{J}$, there are at most $2^n$ ways to pick which of the steps are slow and, on each of the $\ge n/2$ slow steps, there are $\le t^{(2-5\delta)n}$ ways to color $f_{j_r}$ because, throughout the process, $P(f_{J_{r-1}}) \ge P(f_J) \ge t^{(1-2\delta)n^2}$. This would prove \eqref{eq:goal_J}, as desired.

\vspace{3mm}

\noindent {\bf 3. Counting slow steps.} It remains to prove \cref{claim:slow}. Fix $r$ satisfying $P(f_{J_{r-1}}) \ge t^{(1-2\delta)n^2}$. We are interested in counting the number of ``slow colorings'' $f_{j_r}:(I\times \{j_r\})\cup(\{j_r\}\times K) \to [t]$ for which $S(f_{J_{r-1}}) - S(f_{J_r}) \le 2nt$.  
That is, we wish to count the number of choices of $f_{j_r}$ for which
\[
\sum_I \sum_K (f^*_{J_{r-1}}(i,k) - f^*_{J_{r}}(i,k)) \le 2nt.
\]
Since $\sum_K (f^*_{J_{r-1}}(i,k) - f^*_{J_{r}}(i,k))$ is nonnegative, at least $3n/4$ values of $i$ satisfy
\begin{equation}\label{eq:markov1}
\sum_K (f^*_{J_{r-1}}(i,k) - f^*_{J_{r}}(i,k)) \le 8t.
\end{equation}
Expanding out the assumption $P(f_{J_{r-1}}) \ge t^{(1-2\delta)n^2}$, we have
\[
\prod_I \prod_K f^*_{J_{r-1}}(i,k) \ge t^{(1-2\delta)n^2}.
\]
Since $f^*_{J_{r-1}}(i,k) \le t$ for all $i,k$, at least $3n/4$ values of $i$ satisfy
\begin{equation}\label{eq:markov2}
\prod_K f^*_{J_{r-1}}(i,k) \ge t^{(1-8\delta)n}.
\end{equation}
Thus, there must exist a set $I'\subseteq I$ of exactly $n/2$ values of $i$ for which \eqref{eq:markov1} and \eqref{eq:markov2} hold simultaneously. Let us now fix such an $I'$ of size $n/2$ and  count the number of $f_{j_r}$ satisfying \eqref{eq:markov1} for all $i \in I'$. Call such colorings \textit{$I'$-slow}.

Our goal is to prove that the number of $I'$-slow colorings of $f_{j_r}$ is at most $t^{(2-\epsilon/8)n}$. This will finish the proof of Claim 1 as, summing over the at most $2^n$ choices of $I'$, the number of ways to make a slow coloring $f_{j_r}$ is at most $t^{(2-\eps/10)n} = t^{(2-5\delta)n}$, as desired.

 Choosing $f_{j_r}$ is the same as choosing $f_{\{j_r\} \times K}$ and $f_{I\times \{j_r\}}$. Suppose we have already picked $f_{\{j_r\}\times K}$, the colors incident to $K$. Then the remaining pairs from $j_r$ to $I$ can be colored independently: those outside of $I'$ in $t$ ways each and those incident to $I'$ so as to satisfy \eqref{eq:markov1}. For a particular $i\in I'$, let $C_i (f_{\{j_r\}\times K})$ be the set of colors consisting of those $c = f_{j_r}(i,j_r)$ satisfying \eqref{eq:markov1}. Such a color $c$ has to be chosen so that for at most $8t$ of the $n$ values of $k$, $f^*_{J_{r}}(i,k)$ decreases.  

\begin{claim}\label{claim:jr}
    If $i$ satisfies \eqref{eq:markov2}, then the number of choices of $f_{\{j_r\}\times K}$ for which $|C_i (f_{\{j_r\}\times K})| \ge t^{1-\eps}$ is at most $2^t \cdot \binom{n}{8t} \cdot t^{(1-\eps/2)n} \le t^{(1-\eps/4)n}$.
\end{claim}

\begin{proof}
    Fix a choice of $X = C_i (f_{\{j_r\}\times K})$ in at most $2^t$ ways 
     and a choice of the at most $8t$ values of $k$ for which $f^*_{J_{r}}(i,k)$ decreases in at most $\binom{n}{8t}$ ways. Once these choices are fixed, the color choices for $j_r\times k$ can be made independently for each $k$, so long as, for those $k$ outside of the special $8t$, we do not eliminate any possible colors across $(i,k)$. Let $Y_k$ denote the set of valid color choices for $j_r\times k$, so that $|Y_k| = t$ for those in the special set of $8t$.
    
    By \eqref{eq:markov2}, we see that, for at least $3n/4$ values of $k$, we have a set $Z_k\subseteq [t]$ of size $f^*_{J_{r-1}}(i,k) \ge t^{1-\eps}$ of possible colors for $ik$. For all except $8t$ of these, none of these colors can be eliminated by choosing any of the possible colors $ij$ from $X$ and so, for each of at least $3n/4 - 8t \ge n/2$ choices of $k$, we have that there is no triple $(x,y,z) \in X\times Y_k \times Z_k$ for which $g(x,y,z)=\red$. Since $X$ and $Z_k$ are both at least $t^{1-\eps}$, we find that $Y_k$ is less than $t^{1-\eps}$. Thus, the number of choices of $f_{\{j_r\}\times K}$ for this particular choice of $X$ and the exceptional $8t$ values of $k$ is at most $(t^{1-\eps})^{n/2} \cdot t^{n/2}$, so the claim follows.
\end{proof}

This claim essentially finishes the proof. We have
\[
|\{I'\textnormal{-slow colorings }f_{j_r}\}| = \sum_{f_{\{j_r\}\times K}}\prod_i|C_i (f_{\{j_r\}\times K})|,
\]
a sum that has $t^n$ total terms, where each summand is a product of factors which lie in $\{0,1,\ldots, t\}$. We can partition the products above based on whether $i \in I'$, so that, using the trivial bound $t$ for $|C_i (f_{\{j_r\}\times K})|$ when $i \not \in I'$, we obtain
$$\sum_{f_{\{j_r\}\times K}}\prod_i|C_i (f_{\{j_r\}\times K})|
\le \sum_{f_{\{j_r\}\times K}}\left(\prod_{i\in I'}|C_i (f_{\{j_r\}\times K})|\right)t^{n-|I'|}.
$$
 By \cref{claim:jr} and the definition of $I'$, we have that if $i\in I'$, then the corresponding factor $|C_i(f_{\{j_r\}\times K})|$ is at least $t^{1-\eps}$ for at most $t^{(1-\eps/4)n}$ choices of $f_{\{j_r\}\times K}$. This allows us to partition the sum as 
$$\sum_{f_{\{j_r\}\times K}}\prod_{i\in I'}|C_i (f_{\{j_r\}\times K})| =
\sum_{|C_i(f_{\{j_r\}\times K})|\ge t^{1-\eps}}\prod_{i \in I'}|C_i (f_{\{j_r\}\times K})| +
\sum_{|C_i(f_{\{j_r\}\times K})|< t^{1-\eps}}\prod_{i \in I'}|C_i (f_{\{j_r\}\times K})|.$$
The first sum can be bounded by $t^{(1-\eps/4)n} \cdot t^{|I'|}$ and the second sum can be bounded by $t^n \cdot (t^{(1-\eps)})^{|I'|}$. Multiplying by $t^{n-|I'|}$ and recalling that $|I'|=n/2$, this yields
\[
|\{I'\textnormal{-slow colorings }f_{j_r}\}| \le t^{n-|I'|}( t^{(1-\eps/4)n} \cdot t^{|I'|}+t^n \cdot (t^{(1-\eps)})^{|I'|})
= t^{(2-\eps/4)n} + t^{(2-\eps/2)n} <  t^{(2-\eps/8)n},
\]
so the proof is complete.
\end{proof}

\subsection{Avoidability proofs}\label{sec:mpaircomputation}

In this section, we calculate the pair densities of tight cycles and links of cycles, as a result showing that non-tripartite tight cycles and links of non-bipartite graphs are avoidable. Write $C_n^{(3)}$ for the 3-uniform tight cycle of length $n$. Concretely, we can take $V(C_n^{(3)}) = \mathbb Z_n$ and $E(C_n^{(3)})=\{\{k, k+1, k+2\}: k \in \mathbb Z_n\}$.

\begin{prop}\label{tightcycleprop}
    \[\mpair(C_n^{(3)})=\begin{cases} 1/2 & 3\nmid n, \\ 1/3 & 3|n.
    \end{cases}\]
\end{prop}

\begin{proof}
    As  any subgraph $H$ of $C_n^{(3)}$ satisfies $e(H) \le |\partial H|/2$, we can see that $\mpair(C_n^{(3)})\leq 1/2$ by mapping each element of $\partial C_n^{(3)}$ to a distinct vertex.

    If $3|n$, it is not difficult to see that $C_n^{(3)}$ is $123$-inducible and thus has $\mpair=1/3$. It therefore  suffices to show that if $\mpair(C_n^{(3)})<1/2$, then $3|n$.
    
    Take $n$ with $\mpair(C_n^{(3)})<1/2$.  By definition,  we can take an ordering of $V(C_n^{(3)})$, a nonnegative integer $k$ and a surjection $f:\partial C_n^{(3)}\to [k]$ such that $|f(E(C_n^{(3)}))|<k/2$. Consider the $3$-graph $G$ on $[k]$ with edges given by $f(e)$, $e\in E(C_n^{(3)})$. 
    Then $v(G)=k$ and $e(G)<k/2$. Note that $G$ is connected, as in a tight cycle we can go from any edge to any other by taking a sequence of edges any adjacent pair of which share an element of the shadow.

    \begin{lemma}\label{treelemma}
        For any connected $3$-graph $G$,
        \[v(G)\leq 2e(G)+1,\]
        with equality if and only if $G$ is Berge-acyclic. Moreover, in the case of equality, the shadow of $G$ has no cycle of length at least four.
    \end{lemma}
    
    \begin{proof}
        Note that if $c(G)$ is the number of connected components in $G$, then $2e(G)+c(G)-v(G)$ cannot decrease upon adding an edge $e$ to $G$, as easily seen by a short case analysis on the number of new vertices in $e$. For the empty $3$-graph, this quantity is $0$, so, for a connected $3$-graph, $2e(G)+1-v(G)=2e(G)+c(G)-v(G)\geq 0$, as desired. Moreover, if $G'$ is a Berge cycle within $G$, then $2e(G') + c(G') - v(G') \ge 1$, so, again adding edges one at a time, we may conclude that $2e(G)+1-v(G)=2e(G)+c(G)-v(G)\geq 1$. 

        Conversely, suppose that $G$ is Berge-acyclic. If every edge of $G$ has at least $2$ vertices each contained in another edge, then we can create a walk of the form $e_1v_1e_2v_2\cdots$, where $v_i\in e_i, e_{i+1}$ and $e_i\neq e_{i+1}$, $v_i\neq v_{i+1}$. 
        If we get to a vertex or an edge that we have seen before, we obtain a Berge cycle, a contradiction. Thus, $G$ must contain at least one edge $e$ that contains at most one vertex that is contained in any other edge. Removing $e$ decreases $e(G)$ by one and $v(G)$ by $2$, but preserves Berge-acyclicity and connectedness, so a simple induction shows that $v(G)=2e(G)+1$.

        Finally, suppose that $G$ is Berge-acyclic and assume, for the sake of contradiction, that  $\partial G$ has a cycle of length at least $4$. Let $e_1,\ldots,e_n$ be the edges in a shortest such cycle of $\partial G$. If $n\geq 5$, $e_i$ and $e_j$ cannot be in the same edge  $e \in E(G)$ for any $i,j$, as if they were then you could replace them with the third edge of $\partial G$ within $e$ and make a cycle of length $n-1$. Thus, $e_1,\ldots,e_n$ are all in different edges of $G$ and form a Berge cycle, a contradiction.
        If $n=4$, either $e_1,\ldots,e_4$ are in different edges of $G$ and form a Berge $4$-cycle or two of them are in the same edge  $e' \in E(G)$, say $e_1$ and $e_2$. Replacing $e_1$ and $e_2$ with the third element of $\partial G$ within  $e'$, we either obtain a Berge $3$-cycle or the remaining three edges of $\partial G$ are all contained in one edge of $G$. In this latter case, however, we have two edges of $G$ spanning only $4$ vertices, yielding a Berge $2$-cycle, a contradiction. 
    \end{proof}
    
    Applying the lemma, since $v(G)>2e(G)$, we must have that $G$ is Berge-acyclic and that the shadow of $G$ only has cycles of length $3$.

    Consider now the closed walk $W$ of length $n$ on the complete graph with $2k$ vertices $[k]\times\{+,-\}$ defined as follows. Label the vertices of $C_n^{(3)}$ as $v_1,\ldots,v_n$ cyclically. If the edge $v_iv_{i+1}$ has color $j$ (i.e., it is mapped to $j$ under $f$), then we map it to $(j,+)$ if $v_i<v_{i+1}$ in our ordering on $V(C_n^{(3)})$ and $(j,-)$ otherwise. The moves in the walk $W$ go from the image of $v_iv_{i+1}$ to the image of $v_{i+1}v_{i+2}$ (with indices considered cyclically).

    \begin{lemma}\label{walktriangleslemma}
        The walk $W$ has no cycles of length more than $3$ and no edges that are traversed in both directions.
    \end{lemma}
    
    \begin{proof}
        Consider the map $p$ from $[k]\times\{+,-\}\to [k]$ that ignores the sign. Edges in $W$ are sent to edges in the shadow of $G$, so all cycles in the image of $W$ must have length $3$.

        Suppose now, for the sake of contradiction, that $W$ has a cycle $C$ of length at least $4$. The image $p(C)$ is a closed walk in the shadow of $G$. The minimum closed subwalk of $p(C)$ coming from a subpath of $C$ must have length $2$ or $3$, as the shadow of $G$ only has cycles of length $3$. 
        Since $C$ has length greater than $3$, this must come from a subpath of $C$ of the form
        \[(j_1,s_1),(j_2,s_2),(j_1,-s_1)\]
        or
        \[(j_1,s_1),(j_2,s_2),(j_3,s_3),(j_1,-s_1)\]
        for some $j_i\in [k]$, $s_i\in \{+,-\}$.
        We now show that both of these cases give contradictory conclusions about the vertex ordering on $v_1,\ldots,v_n$.

        \vspace{3mm}
        \noindent
        {\it Case 1:} In the first case, there are two edges $v_iv_{i+1}v_{i+2}$ and $v_jv_{j+1}v_{j+2}$ of $C_n^{(3)}$ such that $f(v_iv_{i+1})=f(v_{j+1}v_{j+2})=j_1$ and $f(v_{i+1}v_{i+2})=f(v_jv_{j+1})=j_2$. Thus, in $G$, the images of $v_iv_{i+1}v_{i+2}$ and $v_jv_{j+1}v_{j+2}$ share two vertices and so must be the same edge by Lemma \ref{treelemma}. This implies that $f(v_iv_{i+2})=f(v_jv_{j+2})=j_3$ for some color $j_3$. 
        Since $f$ cannot send two edges to two different permutations of the same triple, we see that the relative orderings of $v_iv_{i+1}v_{i+2}$ and $v_{j+2}v_{j+1}v_j$ must be the same. However, since the sign $s_2$ of $v_{i+1}v_{i+2}$ is the same as the sign of $v_{j}v_{j+1}$, this cannot be the case.

        \vspace{3mm}
        \noindent
        {\it Case 2:} In the second case, consider the edges $v_iv_{i+1}v_{i+2}$, $v_jv_{j+1}v_{j+2}$ and $v_kv_{k+1}v_{k+2}$ that form the three edges of our cycle in $p(C)$. Then we have that \[f(v_iv_{i+1})=f(v_{k+1}v_{k+2})=j_1,\]
        \[f(v_jv_{j+1})=f(v_{i+1}v_{i+2})=j_2\]
        and
        \[f(v_kv_{k+1})=f(v_{j+1}v_{j+2})=j_3.\]
        Thus, each pair of the edges $f(v_iv_{i+1}v_{i+2})$, $f(v_iv_{i+1}v_{i+2})$ and $f(v_iv_{i+1}v_{i+2})$ in $G$ must share a vertex. No two different edges can share two vertices in $G$, so, since this cannot form a loose $3$-cycle, it must be that these three edges of $G$ are the same. Thus, $f(v_jv_{j+2})=j_1$, $f(v_kv_{k+2})=j_2$ and $f(v_iv_{i+2})=j_3$. Once again, since $f$ cannot send different edges to different permutations of the same triple, this gives that the relative orderings of $v_iv_{i+1}v_{i+2}$, $v_{j+2}v_jv_{j+1}$ and $v_{k+1}v_{k+2}v_k$ are identical. However, since $v_iv_{i+1}$ maps to $(j_1,s_1)$ but $v_{k+1}v_{k+2}$ maps to $(j_1,-s_1)$, we see that $v_iv_{i+1}$ and $v_{k+1}v_{k+2}$  have different orderings, contradicting that $v_iv_{i+1}v_{i+2}$ and $v_{k+1}v_{k+2}v_k$ have the same relative ordering. This concludes the proof that $W$ has no cycles of length more than $3$.

        \vspace{2mm}
        Finally, we show that no edge of $W$ is traversed in both directions. Suppose, for the sake of contradiction, that there is an edge $(j_1,s_1),(j_2,s_2)$ that is traversed in both directions. Suppose that in the forward direction this edge comes from $v_iv_{i+1}v_{i+2}$ and in the reverse direction it comes from $v_jv_{j+1}v_{j+2}$. Then $f(v_iv_{i+1})=f(v_{j+1}v_{j+2})=j_1$ and $f(v_jv_{j+1})=f(v_{i+1}v_{i+2})=j_2$. As in the arguments above, we must have that $f(v_iv_{i+2})=f(v_jv_{j+2})$ and $v_iv_{i+1}v_{i+2}$ has the same relative ordering as $v_{j+2}v_{j+1}v_j$. However, this contradicts the fact that $v_iv_{i+1}$ and $v_{j+1}v_{j+2}$ have the same relative ordering (given by $s_1$). Thus, no edge of $W$ is traversed in both directions, finishing the proof of Lemma \ref{walktriangleslemma}.
    \end{proof}
    
    It is an easy induction to show that any closed walk satisfying the conditions of Lemma \ref{walktriangleslemma} has length divisible by $3$ (it must have at least one triangle and removing it gives a closed walk of length $n-3$ also satisfying the conditions). Since $W$ has length $n$, we must have $3|n$, finishing the proof of Proposition \ref{tightcycleprop}.
\end{proof}

We now prove a similar statement for links of cycles. Write $L_n^{(3)}$ for the 3-graph that is the link of the cycle $C_n$. Concretely, $V(L_n^{(3)}) = \mathbb Z_n \cup \{x\}$ and $E(L_n^{(3)})=\{\{x,k,k+1\}: k \in \mathbb Z_n\}$.

\begin{prop}\label{linkcycleprop}
    \[\mpair(L_n^{(3)})=\begin{cases} 1/2 & 2\nmid n, \\ 1/3 & 2|n.
    \end{cases}\]
\end{prop}
\begin{proof}
    The proof at first proceeds along similar lines to the previous one, but there are differences in the details. First, since any subgraph $H$ of $L_n^{(3)}$ satisfies $e(H) \le |\partial H|/2$, we must have $\mpair(L_n^{(3)})\leq 1/2$ for all $n$. If $2|n$, it is easy to see that $L_n^{(3)}$ is tripartite and thus  $\mpair(L_n^{(3)})=1/3$. It therefore remains to show that $\mpair(L_n^{(3)})\geq 1/2$ for all odd $n$.

    Suppose, for the sake of contradiction, that we have $n$ with $\mpair(L_n^{(3)})<1/2$. We must show that $n$ is even. We may take an ordering of $V(L_n^{(3)})$, a nonnegative integer $k$ and a  surjection $h:\partial L_n^{(3)}\to [k]$ such that $|h(E(L_n^{(3)}))|<k/2$. As before, let $G$ be the 3-graph with edges given by $h(e), e\in E(L_n^{(3)})$. Since $G$ is a connected $3$-graph with $v(G)=k$ and $e(G)<k/2$, \cref{treelemma} implies that $G$ is Berge-acyclic.

    Consider now the closed walk $W$ of length $n$ on the complete graph with $2k$ vertices $[k]\times\{+,-\}$ defined as follows. Label the center of $L_n^{(3)}$ as $v$ and the remaining vertices cyclically as $v_1,\ldots,v_n$. If  $vv_i \in \partial L_n^{(3)}$ has color $j$, map it to $(j,+)$ if $v<v_i$ in our ordering on $V(L_n^{(3)})$ and $(j,-)$ otherwise. (Note that this is different from how our mapping uses the ordering in the previous proof!) The moves in the walk $W$ go from the image of $vv_i$ to the image of $vv_{i+1}$ (with indices considered cyclically).

    The equivalent of \cref{walktriangleslemma} in this case is the following.
    \begin{lemma}\label{lem:linkbipartite}
        The walk $W$ is bipartite.
    \end{lemma}

    \begin{proof}
     Suppose, for the sake of contradiction, that $W$ has an odd cycle $C$. Consider the map $p:[k]\times\{+,-\}\to [k]$ that ignores the sign. Edges in $C$ are sent to edges in the shadow of $G$, so all cycles in $p(C)$ must have length $3$. Thus, in the graph $p(C)$, all minimal closed subwalks have length $2$ or $3$. So we may partition the edges of $p(C)$ (which possibly occur with repetition) into closed subwalks of lengths $2$ or $3$.

        Call an edge of $C$ (and the corresponding edge of $p(C)$) \emph{switching} if it goes from $+$ to $-$ or $-$ to $+$. Since $C$ is a cycle, it contains an even number of switching edges. As $p(C)$ has odd length, one of the closed subwalks in our partition must have a number of switching edges of different parity to its length. Thus, $p(C)$ contains one of the following:
        \begin{enumerate}
            \item A closed walk of length $2$ using exactly $1$ switching edge.
            \item A closed walk of length $3$ using $0$ or $2$ switching edges.
        \end{enumerate}
        We now show that both of these are impossible.
        
        \vspace{3mm}
        \noindent
        {\it Case 1:} Suppose we have a closed walk of length $2$ using exactly $1$ switching edge. We may assume the walk goes $a\to b\to a$ for some $a,b\in [k]$.
        This implies that there are indices $i,j$ such that $vv_i$ and $vv_{j+1}$ are colored $a$, $vv_j$ and $vv_{i+1}$ are colored $b$. 
        Now $v_iv_{i+1}$ and $v_jv_{j+1}$ must be colored the same color, call it $c$, or else $G$ would have two distinct $3$-edges sharing two vertices. Thus, $vv_iv_{i+1}$ and $vv_jv_{j+1}$ must correspond to the same edge in $G$. Comparing which $2$-edges are colored $a,b,c$, we see that this implies that $vv_iv_{i+1}$ and $vv_{j+1}v_j$ have the same relative ordering. 
        But now note that the edge in our walk corresponding to $vv_iv_{i+1}$ is switching if and only if $v$ lies between $v_i$ and $v_{i+1}$. Thus, the edge corresponding to $vv_iv_{i+1}$ is switching if and only if the edge corresponding to $vv_jv_{j+1}$ is, meaning it is impossible for exactly one such edge to be switching.

        \vspace{3mm}
        \noindent
        {\it Case 2:} Suppose we have a closed walk of length $3$ using $0$ or $2$ switching edges. We may assume the walk goes $a\to b\to c\to a$ for some $a,b,c\in [k]$. This implies that there are indices $h,i,j$ such that 
        \begin{itemize}
            \item $vv_h$ and $vv_{j+1}$ are colored $a$,
            \item $vv_i$ and $vv_{h+1}$ are colored $b$,
            \item $vv_j$ and $vv_{i+1}$ are colored $c$.
        \end{itemize}
        As this cannot yield a Berge $3$-cycle in $G$ or two $3$-edges sharing two vertices, we must have that $v_hv_{h+1}$ is colored $c$, $v_iv_{i+1}$ is colored $a$ and $v_jv_{j+1}$ is colored $b$. As in the previous case, this implies that $vv_hv_{h+1}$, $v_iv_{i+1}v$ and $v_{j+1}vv_j$ have the same relative ordering. Once again, the edge of our walk corresponding to $vv_iv_{i+1}$ is switching if and only if $v$ lies between $v_i$ and $v_{i+1}$ in the ordering (and similarly for $j$ and $k$), so exactly one of our three edges must be switching (depending on which is the middle element in the ordering). Thus, it is impossible for $0$ or $2$ switching edges to appear, finishing the proof.
    \end{proof}
    Now \cref{linkcycleprop} immediately follows, as \cref{lem:linkbipartite} implies that $W$ is a closed walk in a bipartite graph and thus its length $n$ must be even.
\end{proof}

\subsection{Proof of \cref{thm:nlognsupersaturation}} \label{sec:supersat}

In this section, we prove \cref{thm:nlognsupersaturation}, which says that \cref{thm:mainmpair} is tight for a large class of $3$-graphs, namely, any $3$-graph that can be created by starting with the empty graph and iteratively adding edges that strictly increase the number of edges in the shadow of $H$ at each step. 
This theorem will follow immediately from the following two lemmas, both of which are proved using supersaturation techniques.

\begin{lemma}\label{lem:K1nn}
    For all $3$ graphs $H$, $r(H,K_{n,n,n}^{(3)})\leq r(H,K_{1,n,n}^{(3)})^{3n}$.
\end{lemma}

\begin{lemma}\label{lem:1skeleton}
    If $H$ is a $3$-graph such that $r(H,K_{1,n,n}^{(3)})=n^{O(1)}$ and $x,y,z$ are vertices such that at least one of $xy$, $yz$ and $zx$ is not in the shadow of $H$, then
    \[r(H\cup\{xyz\}, K_{1,n,n}^{(3)}) = n^{O(1)}.\]
\end{lemma}
 
To derive \cref{thm:nlognsupersaturation}, observe that if $H$ satisfies the conditions of that theorem, we can apply \cref{lem:1skeleton} repeatedly to get $r(H,K_{1,n,n}^{(3)})=n^{O(1)}$ and  \cref{lem:K1nn} then implies that $r(H,K_{n,n,n}^{(3)})\leq n^{O(n)}$. It will therefore suffice to prove the two lemmas.

\begin{proof}[Proof of \cref{lem:K1nn}]
    Fix $n$ and let $A=r(H,K_{1,n,n}^{(3)})$. Let $\chi$ be a $2$-coloring of the complete $3$-graph on $A^{3n}$ vertices and suppose that $\chi$ has no red copy of $H$.

    Taking a random induced subgraph on $A$ vertices, as there is no red copy of $H$, there must be a blue copy of $K_{1,n,n}^{(3)}$. Since every copy of $K_{1,n,n}^{(3)}$ has probablity at most
    \[\frac{\binom{A^{3n}-(2n+1)}{A-(2n+1)}}{\binom{A^{3n}}{A}}= \frac{\binom{A}{2n+1}}{\binom{A^{3n}}{2n+1}}\leq A^{-(2n+1)(3n-1)}<A^{-6n^2-1}\]
    of being chosen, there must be at least $A^{6n^2+1}$ blue copies of $K_{1,n,n}^{(3)}$. Thus, at least
    \[\frac{A^{6n^2+1}}{\binom{A^{3n}}{2n}\binom{2n}{n}}>A\geq n\]
    copies of $K_{1,n,n}^{(3)}$ must use the same underlying $K_{n,n}$, yielding a copy of $K_{n,n,n}^{(3)}$ in blue.
\end{proof}
\begin{proof}[Proof of \cref{lem:1skeleton}]
   By adding isolated vertices if necessary, we can and will assume that $\{x,y,z\} \subseteq V(H)$. 
   Suppose that $r(H,K_{1,n,n}^{(3)})=A$. Assume, without loss of generality, that $xy$ is not in the shadow of $H$. Let $\chi$ be a $2$-coloring of the complete graph on $A^{2v(H)}$ vertices and suppose that $\chi$ has no blue copy of $K_{1,n,n}^{(3)}$. 

    Taking a random induced subgraph on $A$ vertices, as there is no blue copy of $K_{1,n,n}^{(3)}$, there must be a red copy of $H$. Since every copy of $H$ has probability 
    \[\frac{{A^{2v(H)}-v(H) \choose A-v(H)}}{{A^{2v(H)} \choose A}} =\frac{\binom{A}{v(H)}}{\binom{A^{2v(H)}}{v(H)}}\leq A^{v(H)-2v(H)^2}\]
    of being chosen, $\chi$ must have at least $A^{2v(H)^2-v(H)}$ red copies of $H$. Thus, there must be at least
    \[A^{2v(H)^2-v(H)-2v(H)(v(H)-2)}=A^{3v(H)}\]
    copies of $H$ in red that have all vertices except $x$ and $y$ in a fixed location. Let $F$ be the set consisting of these copies. Let $z_0$ be the location where $z$ appears in these copies and let $S$ and $T$ be the possible locations for $x$ and $y$ among these copies. Then $|S||T|\geq A^{3v(H)}$, so $|S|,|T|\geq A^{v(H)}>2n$. Since $\chi$ contains no blue $K_{1,n,n}$, there must be a red edge with one vertex at $z_0$ and one vertex each in $S$ and $T$, say at $s$ and $t$.

    Since $xy$ is not in the shadow of $H$, the positions of $x$ and $y$ in $S$ and $T$ are independent after the positions of the remaining vertices are fixed, so there must be a copy of $H$ in $F$ such that $x$ and $y$ are sent to $s$ and $t$. Thus, adding the red edge $z_0st$ to this copy, we obtain a copy of $H\cup \{xyz\}$. Therefore,
    \[r(H\cup\{xyz\},K_{1,n,n}^{(3)})\leq r(H,K_{1,n,n}^{(3)})^{2v(H)},\]
    as required.
\end{proof}

\section{Stepping up for linear hypergraphs versus cliques}\label{sec:steppingupintro}

In this section, we prove \cref{thm:linearsteppingup}, which states that for all sufficiently large $k$ there exists a linear $3$-graph $F$ on $k$ vertices such that
\[
r(F, K_{n}^{(3)}) > 2^{c(\log n)^{\sqrt{k}/(128\log^5k)}}, 
\]
where $c > 0$ depends only on $k$, though we note that we have not made a serious attempt at optimizing the power on the log. The proof divides into two parts: first constructing a ``quasirandom" $F$ which has many edges between certain appropriately chosen vertex subsets and then using the stepping-up technique to give a coloring with no red copy of $F$ and no blue copy of $K_n^{(3)}$. 
These objectives are accomplished separately in the next two subsections.

\subsection{Constructing $F$}

In this section, we make use of a probabilistic technique pioneered by Krivelevich~\cite{Kr} to show that there are linear hypergraphs with certain ambient properties that will be needed for our stepping-up argument.

\begin{lemma}\label{steiner}
For all sufficiently large  $k$, there is a linear triple system $F$ on $k$ vertices with the following three properties:

\begin{enumerate}

\item Every vertex subset of order $k^{3/5}$ has at least one edge.

    \item For any two disjoint subsets $S,T\subset V(F)$ with $|S\cup T| > k/4$, $|S| > \log^2 k$ and $|T| > k^{1/2}\log k$, there is at least one edge with one vertex in $S$ and two vertices  in $T$.

    \item For any disjoint vertex subsets $A_1,\ldots, A_t,S \in V(F)$ with $|S| > k/4$, $|A_i| \le k^{1/2}\log k$, $|A_1\cup \cdots \cup A_t| \ge k/(16\log k)$, there are at least 
    $$\frac{|S|}{500k} \sum\limits_{i,j}|A_i||A_j|$$
    \noindent edges in $F$ with the property that one vertex is in $S$ and the other two are in distinct parts $A_i$ and $A_j$.  

\end{enumerate}

\end{lemma}

\begin{proof}
    Consider the binomial random $3$-graph $H \sim H^{(3)}(k, p)$ with $k$ vertices where each edge appears independently with probability $p=1/200k$. Let $B$ denote the unique $3$-graph with $4$ vertices and $2$ edges and let $\mathcal{B}$ denote a maximal collection of edge-disjoint copies of $B$ in $H$. Form $F$ by starting with $H$ and deleting both edges from every copy of $B$ in $\mathcal{B}$. Then, by the maximality of $\mathcal{B}$, the remaining $3$-graph $F$ has no copies of $B$ and, hence, is a linear triple system. We will now show that with high probability $F$ has the three required properties.  We start with the second and third properties and then observe that the first property can be proved similarly.

    Pick $S, T \subset V(F)$ as in the second property of the lemma. Call an edge in $H$ with one vertex in $S$ and two vertices in $T$ an $STT$ edge.
    Let $X=X_{S,T}$ be the number of $STT$ edges, let $Y=Y_{S,T}$ be the number of copies of $B$ that contain at least one $STT$ edge and let $Z=Z_{S,T}$ be the maximal number of pairwise edge-disjoint copies of $B$, each containing at least one $STT$ edge. Obviously, $Z \le Y$. Define the event $$A=A_{S,T}= \{X> 10Z\}.$$ We note that if $A_{S,T}$ holds for every appropriate $S, T$, then $F$ satisfies the second property. Indeed, if we delete all edges in all copies of $B$ in $\mathcal{B}$, then the number of $STT$ edges that are deleted is at most $2Z$, leaving at least one $STT$ edge.

 Set $s=|S|$ and $t=|T|$.  Since $\EE X=ps{t \choose 2}$ and $p=1/200k$,  
  $$\frac{1}{200} \EE X = p^2ks{t \choose 2}\le \EE Y\le 3p^2ks{t \choose 2}= \frac{3}{200}\EE X.$$
 Consequently, 
$a \, \EE  Y \le \EE X \le b \, \EE Y$ for $a=200/3>60$ and $b=200$. Now
$$\Pr(\overline{A}) =\Pr(X \le 10Z) \le
\Pr\left(X \le  \frac{\EE X}{2}\right) +\Pr\left(Z \ge \frac{\EE X}{20}\right)\le 
\Pr\left(X \le  \frac{\EE X}{2}\right) +\Pr\left(Z \ge \frac{a\EE Y}{20}\right).
$$
  Krivelevich~\cite[Claim 1]{Kr} proved that, for any constant $c>0$,
$$\Pr(Z\ge c \, \EE Y) < e^{-c\,(\log c-1)\EE Y}$$
and, therefore,
$$\Pr\left(Z \ge \frac{a\EE Y}{20}\right)\le 
e^{-\frac{a}{20}(\log(a/20)-1)\EE Y}
\le
e^{-\frac{a}{20b}(\log(a/20)-1)\EE X}.$$
Note that $\log(a/20)>1$, since $a>60$. Moreover, the standard Chernoff bound gives that
$$\Pr\left(X \le  \frac{\EE X}{2}\right) < e^{-\EE X/8}.$$
Recall that $\EE X=ps{t \choose 2}=s{t \choose 2}/200k$. The number of choices for $S,T$ of orders $s,t$ is at most ${k \choose s}{k \choose t} <k^{s+t}$. Thus, in order to show that $\sum_{S,T}\Pr(\overline{A_{S,T}})=o(1)$, the inequalities above imply that we just need 
$st^2/k \gg (s+t) \log k$. Note that $st^2/k \gg s\log k$, since $t > \sqrt{k}\log k$.  Moreover, $st^2/k \gg t\log k$, since $s + t > k/4$, $s> \log^2k$ and $t> \sqrt{k}$. Therefore, $F$ does indeed satisfy the second property with high probability.

We now turn to the third property. Writing $X' = X'_{S,T}$ for the number of $STT$ edges in $F$, we first observe  that for all $S, T$ with $|S| >k/4$ and $|T|\ge k/(16 \log k)$, w.h.p.~$X' > \EE X'/2$ in $F$. Indeed, the previous argument shows that w.h.p.~the number of $STT$ edges in $F$ is at least $4X'/5 > \EE X'/2$, where, in the inequality, we used that, by the Chernoff bound, $X'>3\EE X'/4$ w.h.p.  

Now suppose that we are given $S$ and $A_1, \ldots, A_t$ as in the third property. Set $T = \cup_i A_i$. 
Since $F$ is linear and $|A_i|\le k^{1/2} \log k$ for all $i$, the number of $SA_iA_i$ edges in $F$ over all $i$ is at most 
\begin{equation} \label{ai2}\sum_i |A_i|^2 \le \frac{k}{k^{1/2} \log k} (k^{1/2}\log k)^2 = k^{3/2} \log k. \end{equation}
The $STT$ edges are either $SA_iA_j$ edges for $i \ne j$ or $SA_iA_i$ edges.  Note that ${t \choose 2} > \sum_{i\ne j}|A_i||A_j|$ and $\EE X_{S,T}=ps{t \choose 2}.$ In view of (\ref{ai2}), and $k$ sufficiently large, we have that $\sum_{i\ne j}|A_i|A_j| >k^2/2\log^2 k$.
Writing $X_{SA_iA_i}$ for the number of $SA_iA_i$ edges, we conclude that w.h.p.~the number of $SA_iA_j$ edges over all $i \ne j$ in $F$ is at least 
$$ X'_{S,T}-\sum_i X_{SA_iA_i} 
\ge \frac{\EE X_{S,T}}{2}-k^{3/2}\log k> \frac{s}{400k}\sum\limits_{i\ne j}|A_i||A_j| - k^{3/2}\log k
>\frac{s}{500k}\sum\limits_{i\ne j}|A_i||A_j|.$$

Finally, we consider the first property. To this end, suppose $S \subset V(F)$ with $|S|=k^{3/5}$. Define $X=X_S$, $Y=Y_S$, $Z=Z_S$ 
in the obvious way as above and let $A=A_S$ be the event that $X>10Z$. 
Proceeding as we did for the second property and using that $\EE X=p{s \choose 3}$, $\Pr(\overline{A}) = o(1)$ if the condition $ps^3\gg s \log k$ holds. But this is equivalent to $s^2 \gg k \log k$, which holds with plenty of room to spare.

   We have therefore shown that $F$ satisfies each property w.h.p.,  so there is indeed a linear $F$ that satisfies all three properties simultaneously.
 \end{proof}

In our analysis, we will also need the following result, which comes from a standard application of the probabilistic method.

\begin{lemma}\label{cyclecomplete}
For every $k \geq 3$, there is a positive constant $c = c(k)$ such that,  for every $n > 2^k$, there is a graph $G$ on $m =c(n/\log n)^{\frac{\sqrt{k}}{64\log^5k}}$ vertices with independence number less than $n$ where every subset with $r$ vertices, for every $k^{1/2}/(16\log^2k) \le r \le k$, induces fewer than $r^2/\log^3k$ edges.  
\end{lemma}

\begin{proof}

Consider $G(m,p)$ with $p = m^{-\frac{64\log^5k}{\sqrt{k}}}$.  Then, summing over all $r$ with $k^{1/2}/(16\log^2k) \le r \le k$, the expected number of subsets of order $r$ that induce at least $r^2/\log^3k$ edges is at most
$$\sum\limits_{r \ge \frac{k^{1/2}}{16\log^2k} }\binom{m}{r}2^{r^2}p^{r^2/\log^3k} < \sum\limits_{r \ge \frac{k^{1/2}}{16\log^2k} }m^{2r}m^{\frac{-64r^2\log^2k}{\sqrt{k}}} < 1/3.$$

\noindent Moreover, the expected number of independent sets of size $n$ is at most
$$\binom{m}{n}(1-p)^{n\choose 2} \leq m^ne^{-pn^2/2} = e^{n\log m - \frac{n^2}{2m^{64\log^5k/\sqrt{k}}}}  < 1/3,$$

\noindent where the last inequality follows from the fact that $m =c(n/\log n)^{\frac{\sqrt{k}}{64\log^5k}}$ and that $c$ may be taken sufficiently small in terms of $k$.  By Markov’s inequality and the union bound, the statement follows.
\end{proof}

\subsection{Stepping up}\label{sec:steppingupconclusion}

\begin{proof}[Proof of \cref{thm:linearsteppingup}]

Let $k$ be a sufficiently large constant that will be determined later. Let $F$ be the linear triple system on $k$ vertices given by Lemma~\ref{steiner} and let $G$ be a graph on $m =  c(n/\log n)^{\frac{\sqrt{k}}{64\log^5k}}$ vertices, where $V(G) = \{0,1,\ldots, m-1\}$, with the properties described in Lemma~\ref{cyclecomplete}.  Set $V= \{0,1,\ldots, 2^m - 1\}$.  Given subsets $U_1,U_2 \subset V$, we write $U_1 < U_2$ if $u_1 < u_2$ for all $u_1 \in U_1$ and $u_2\in U_2$. In what follows, we will use $G$ to define a coloring $\chi:\binom{V}{3}\rightarrow \{\textnormal{red, blue}\}$ of the triples of $V$ with no red copy of $F$ and no blue clique on $\binom{m + n -1}{n-1}$ vertices.

For each $v \in V$, write $v=\sum_{i=0}^{m-1}v(i)2^i$ with $v(i) \in \{0,1\}$ for each $i$. For $u \not = v$, let $\delta(u,v) \in V(G)$ denote the largest $i$ for which $u(i) \not = v(i)$.  It is easy to verify the following properties (see, for example, \cite{GRS}).

\begin{description}

\item[Property I:] For every triple $u < v < w$, $\delta(u,v) \not = \delta(v,w)$.

\item[Property II:] For $v_1 < \cdots < v_r$, $\delta(v_1,v_{r}) = \max_{1 \leq j \leq r-1}\delta(v_j,v_{j + 1})$.

\end{description}

We define $\chi:\binom{V}{3} \rightarrow \{\textnormal{red, blue}\}$ as follows.  For vertices $v_1 < v_2 < v_3$ in $V$, let $\delta_i = \delta(v_i,v_{i + 1})$ and set $\chi(v_1,v_2,v_3) =$ red if and only if $\delta_1 > \delta_2$ and $\delta_1\delta_2 \in E(G)$. We claim that this coloring does not contain a blue clique of order $\binom{m + n -1}{n-1}$.

\begin{lemma}\label{blueclique}
The coloring $\chi$ does not contain a blue clique on $\binom{m + n -1}{n-1}$ vertices.
\end{lemma}

Lemma \ref{blueclique} will follow from the claim below. To state it, we need some definitions. Given a subset $A\subset V$, we say that $A$ contains a \emph{$\delta$-increasing set of order $s$} if there are vertices $v_1,\ldots, v_s \in A$ such that $v_1 < v_2 < \cdots < v_s$ and

$$\delta(v_1,v_2) < \delta(v_2,v_3) < \cdots < \delta(v_{s-1},v_s).$$

\noindent Likewise, we say that $A$ contains a \emph{$\delta$-decreasing set of order $t$} if there are vertices $v_1,\ldots, v_t \in A$ such that $v_1 < v_2 < \cdots < v_t$ and

$$\delta(v_1,v_2) > \delta(v_2,v_3) > \cdots > \delta(v_{t-1},v_t).$$

\begin{claim}\label{pascal}
Let $A\subset V$ be such that $A$ does not contain a $\delta$-increasing set of order $s$ or a $\delta$-decreasing set of order $t$.  Then $|A| \leq \binom{s + t - 4}{t-2}.$
\end{claim}

\begin{proof}
We proceed by induction on $s$ and $t$.  Let $A = \{v_1,\ldots, v_{|A|}\}$, where $v_1 < v_2 < \cdots < v_{|A|}$. For the base case $s = 3$ and $t \geq 3$, if $A$ does not contain a $\delta$-increasing set of order $3$, then we must have
$$\delta(v_1,v_2) > \delta(v_2,v_3) > \cdots > \delta(v_{|A|-1},v_{|A|}).$$
\noindent Since $A$ does not contain a $\delta$-decreasing set of order $t$, we must have $|A|\leq t-1 = \binom{s + t - 4}{t-2}.$  A symmetric argument works for the other base case $s\geq 3$ and $t = 3$.

For the inductive step assume that the statement holds for $s' < s$ or $t' < t$.  Suppose $A$ does not contain  a $\delta$-increasing set of order $s$ or a $\delta$-decreasing set of order $t$. Let $j \in \{1,\ldots, m-1\}$ be such that
$$\delta_j = \delta(v_j,v_{j + 1}) = \max_i\delta(v_i,v_{i + 1})$$
\noindent and let $A_1 = \{v_1,\ldots, v_j\}$ and $A_2 = \{v_{j + 1},\ldots, v_{|A|}\}$.  Then, by Properties I and II above, $A_1$ does not contain a $\delta$-increasing set of size $s-1$, since otherwise $A_1\cup v_{j + 1}$ contains a $\delta$-increasing set of size $s$, a contradiction.  Likewise, $A_2$ does not contain a $\delta$-decreasing set of size $t-1$, since otherwise Properties I and II would imply that $A_2 \cup v_j$ contains a $\delta$-decreasing set of size $t$, again a contradiction. Therefore, by the induction hypothesis, we have
$$|A| = |A_1| + |A_2| \leq \binom{s + t - 5}{t-2} + \binom{s+t - 5}{t-3} = \binom{s + t-4}{t-2},$$
as required.
\end{proof}

\begin{proof}[Proof of Lemma \ref{blueclique}]
    Let $A\subset V$ be such that $\chi$ colors every triple in $A$ blue.  Clearly, $A$ cannot contain a $\delta$-increasing set of order $m + 2$. Moreover, by Properties I and II above, 
    $A$ does not contain a $\delta$-decreasing set on $n + 1$ vertices, as this would correspond to an independent set of order $n$ in $G$. Therefore, by Claim~\ref{pascal}, we have $|A| \leq \binom{m + n -1}{n-1}.$
\end{proof}

Next, we show that $\chi$ does not contain a red copy of $F$.   For the sake of contradiction, suppose $\chi$ contains a copy of $F$ with vertex set $V' = \{v_1, \ldots, v_k\} \subset V$, where $v_1 < \cdots < v_k$, and let $\delta_i = \delta(v_i,v_{i + 1})$.  
In what follows, we will define a vertex partition  $$\mathcal{P}_t:V' = A_1\cup A_2 \cup \cdots \cup A_r \cup S_t\cup B_s\cup B_{s-1}\cup \cdots \cup B_1$$ such that the following conditions hold:

\begin{enumerate}

\item $t = r + s + 1$.

\item $A_1 < \cdots < A_r < S_t < B_s < \cdots < B_1.$

\item $|A_i| \le k^{1/2}\log k$, $|B_j| \le \log^2k$ and $|S_t| > k/4$.

    \item  There are $\delta^{\ast}_1 > \cdots > \delta^{\ast}_r$ such that for any $u_1 \in A_i$, $u_2 \in A_j$, $u_3 \in S_t$, where $i < j$, we have $u_1 < u_2 < u_3$ and
    $$\delta(u_1,u_2) =  \delta^{\ast}_i > \delta^{\ast}_j = \delta(u_2,u_3).$$

    \item  There are $\tilde{\delta}_1 > \cdots > \tilde{\delta}_s$ such that for any $u_1 \in B_i$, $u_2 \in B_j$, $u_3 \in S_t$, where $i < j$, we have $u_3 < u_2 < u_1$ and $${\delta}(u_3,u_2) =   \tilde{\delta}_j < \tilde{\delta}_i  = \delta(u_2,u_1).$$

    \item $|A_1\cup \cdots \cup A_r| < k/4$ and $|B_1\cup \cdots \cup B_s| < k^{3/4}$.

\end{enumerate}

\noindent

\noindent We start with $\mathcal{P}_1 :V'= S_1$. Suppose we have the partition $$\mathcal{P}_{t-1}:V' = A_1\cup \cdots \cup A_{r} \cup S_{t-1}\cup B_s\cup \cdots \cup B_1,$$ with the properties described above. 
If $|A_1\cup \cdots \cup A_{r}| < \frac{k}{4} - k^{1/2}\log k$ and $|B_1\cup \cdots \cup B_s| < k^{3/4} - \log^2k$, we define $\mathcal{P}_t$ by partitioning $S_{t-1}$ as follows.  

Let $S_{t-1} = \{v_{w}, v_{w + 1}, \ldots, v_{w'}\}$, where $w < w'$.  We define $w \leq z < w'$ such that $$\delta(v_{z},v_{z + 1}) = \max_{w\leq i < w'}\delta(v_i,v_{i + 1}).$$   
Let $T = \{v_w, \ldots, v_z\}$ and $S = \{v_{z+1},\ldots, v_{w'}\}$.  
Note that $|S\cup T| =|S_{t-1}|> k/4$. Suppose $|T| > k^{1/2}\log k$ and $|S| > \log^2k$.  By Properties I and II, $\chi$ colors every triple with two vertices in $T$ and one in $S$ blue.  However, by the second property of $F$, one such triple must be red, which is a contradiction.   Hence, we must have either $|T| \le k^{1/2}\log k$ or $|S| \le \log^2k$.

If $|T| \le k^{1/2}\log k$, then we set $A_{r+ 1} = T = \{v_w,v_{w + 1}, \ldots, v_z\}$ and $S_{t} = S = \{v_{z + 1}, v_{z + 2}, \ldots, v_{w'}\}$ and we have the partition

$$\mathcal{P}_{t} = A_1\cup \cdots \cup A_{r + 1}\cup S_{t} \cup B_s \cup \cdots \cup B_1.$$

\noindent Moreover, for $\delta^{\ast}_{r + 1} = \delta(v_z,v_{z + 1})$, this partition satisfies the required properties.

If $|S| \le \log^2k$, then we set $S_{t} = \{v_w, v_{w + 1}, \ldots , v_z\}$ and $B_{s +1} = \{v_{z + 1}, \ldots, v_{w'}\}$ and we have the partition

$$\mathcal{P}_{t} = A_1\cup \cdots \cup A_{r}\cup S_{t} \cup B_{s + 1} \cup B_s \cup \cdots \cup B_1.$$

\noindent Moreover, for $\tilde{\delta}_{s + 1} = \delta(v_z,v_{z + 1})$, this partition again satisfies the required properties.  

Let $t$ be the maximum integer such that the partition $$\mathcal{P}_t = A_1\cup \cdots \cup A_r\cup S_t\cup B_s \cup \cdots \cup B_1$$ satisfies all six properties described above.  Then $|A_1\cup \cdots \cup A_r| > k/8$ or $|B_1\cup \cdots \cup B_s| > k^{3/4}/2$.  The proof now falls into two cases.

\medskip

\noindent \emph{Case 1: $|A_1\cup \cdots \cup A_r| > k/8$.}

\medskip

By partitioning dyadically and averaging, we can conclude that there is an integer $K$ and $r'$ indices $j_1 < j_2 < \cdots < j_{r'}$ such that

$$2^K \leq |A_{j_i}| < 2^{K + 1}$$

\noindent for all $i \in \{1,\ldots, r'\}$ and

$$\frac{k}{16\log k} \leq \frac{k}{8\log_2 k} \leq  \sum\limits_{i = 1}^{r'}|A_{j_i}|\leq k.$$

\noindent Hence, each part satisfies

$$\frac{k}{32r'\log k}\leq |A_{j_i}| \leq \frac{2k}{r'}.$$

\noindent In what follows, we will show that $R=\{\delta^{\ast}_{j_1},\ldots, \delta^{\ast}_{j_{r'}}\}$ induces (nearly) quadratically many edges, contradicting the properties of $G$.

Since $|S_t| > k/4$, $|A_{j_i}| \le k^{1/2}\log k$ and $|A_{j_1}\cup \cdots \cup A_{j_{r'}}| \ge \frac{k}{16\log k}$, the third property of $F$ implies that there is an absolute constant $c'$ such that $\chi$ colors  $$f \ge \sum\limits_{i,\ell}|A_{j_i}||A_{j_{\ell}}| \frac{|S|}{500k} \geq \frac{k^2}{c'\log^2k}$$ triples, with one vertex in $S$ and the other two in distinct parts $A_{j_i}, A_{j_{\ell}}$, red.
On the other hand, by property 4 of the partition, each such red triple is obtained from an edge  $\delta^{\ast}_{j_i}\delta^{\ast}_{j_{\ell}} \in E(G)$. Since $F$ is linear, the number of such triples going across parts $(A_{j_i},A_{j_{\ell}},S)$, with one vertex in each part, is at most $|A_{j_i}||A_{j_{\ell}}|$. Writing $e_R$ for the number of edges in the induced subgraph $G[R]$, we  obtain
$$f \le \sum_{\delta^{\ast}_{j_i}\delta^{\ast}_{j_{\ell}} \in E(G)} |A_{j_i}||A_{j_{\ell}}| \le e_R  \frac{4k^2}{(r')^2}.
$$
This yields
$$e_R \ge \frac{(r')^2}{4c'\log^2 k}.$$
From
$$\frac{k}{16\log k} \le 
\sum\limits_{i = 1}^{r'}|A_{j_i}| \le r' (k^{1/2} \log k),$$
we obtain $r' \ge k^{1/2}/(16\log^2 k)$.   On the other hand, since $k$ may be assumed to be sufficiently large, the properties of $G$ imply that $e_R \leq (r')^2/(\log^3k) < (r')^2/(4c'\log^2k)$, a contradiction.

\medskip

\noindent \emph{Case 2: $|B_1\cup \cdots \cup B_s| > k^{3/4}/2$.}

\medskip

Since $1\leq |B_i| \le \log^2k$, we have $s > k^{3/5}$.  Let $u_i \in B_i$ and consider the subset $\{u_1,\ldots, u_s\} \subset V'$.  Then, by property 5 of the partition, for $u_i > u_j > u_k$, we have $\delta(u_i,u_j) > \delta(u_j,u_k)$.  Hence, by Properties I and II, every triple in $\{u_1,\ldots, u_s\}$ is blue.  However, by the first property of $F$, there must be at least one red triple, a contradiction.

\medskip

Since $|V| = 2^m = 2^{c(n/\log n)^{\frac{\sqrt{k}}{64\log^5k}}}$, for $n' = \binom{m + n}{n} \leq (m + n)^n \leq m^{2n}$, we have $\log(n') = \Theta_k(n\log n)$ and $\log\log(n') = \Theta_k(\log n$).  Hence, there is a constant $c_0$ depending only on $k$ such that 

$$r_3(F,n') > |V| = 2^m  = 2^{c(n/\log n)^{\frac{\sqrt{k}}{64\log^5k}}} \geq 2^{c_0\left(\frac{\log n'}{(\log\log n')^2}\right)^{\frac{\sqrt{k}}{64\log^5k}}}>
2^{c_0 (\log n')^{\frac{\sqrt{k}}{128\log^5k}}},$$
as required.
   \end{proof}

\section{Concluding remarks} \label{sec:conclusion}

\subsection{Pair constructions} 

All of the best known lower-bound constructions for $3$-uniform hypergraph Ramsey numbers are ``pair constructions'', which can be formally defined as follows.

\begin{definition}
    A \textit{$3$-uniform pair construction} (or, henceforth, pair construction) on vertex set $[N]$ is a coloring $\chi_{f,g}$ of the complete $3$-graph $K_N^{(3)}$ obtained from a pair of functions $f:\binom{[N]}{2} \to [p]$ and $g:[p]^3 \to \{\textnormal{red}, \textnormal{blue}\}$ by the rule $\chi_{f,g}(\{i, j, k\}) = g(f(i,j), f(j, k), f(k, i))$ for any edge $\{i, j, k\} \in \binom{[N]}{3}$ satisfying $i<j<k$. We call the parameter $p$ the \textit{complexity} of the pair construction.
\end{definition}

To see how this definition relates to existing constructions, we note that stepping-up constructions are exactly pair constructions $\chi_{f,g}$ where $f(i,j) = \delta(i,j)$ is the index of the highest significant binary digit where $i$ and $j$ differ and typically have complexity $O(\log N)$. Similarly, inducibility constructions, like those based on random tournaments or $3$-colorings, are pair constructions of constant complexity where $f:\binom{[N]}{2}\to [p]$ is a uniform random function. The two constructions described in this paper also fit firmly into this paradigm.

These observations suggest the following natural question: do all Ramsey colorings of $3$-graphs admit low-complexity encodings? More formally, if, for any $3$-graphs $G$ and $H$, we define $r^p(G,H)$ to be the smallest $n$ such that every $2$-edge-coloring of $K_{n}^{(3)}$ of complexity at most $p$ contains either a red copy of $G$ or a blue copy of $H$, then we have the following broad conjecture.

\begin{conj}
There exists an absolute constant $a > 0$ such that, for all $3$-graphs $G$ and $H$, $r^p(G,H) \ge r(G,H)^a$ whenever $p >r(G,H)^{o(1)}$.
\end{conj}

In other words, we believe that if there is a Ramsey coloring of order $N$, then there is also a Ramsey coloring of order $N^{\Omega(1)}$ with much smaller complexity. This would at least explain the prevalence of such constructions in Ramsey theory, but might also help shed some light on the possible growth rates for $r(H, K_n^{(3)})$, a topic which we now discuss.

While almost all known Ramsey colorings can be encoded with complexity which is polylogarithmic in $r(G,H)$, we now show that such a strong dependence is not possible in general.

\begin{prop}
    For $n, p\ge 1$, $r^p(K_{n,n,n}^{(3)}) \le 2^{O(n \log p)}$. In particular, if $p=2^{o(n)}$, then $r^p(K_{n,n,n}^{(3)}) = r(K_{n,n,n}^{(3)})^{o(1)}$.
\end{prop}
\begin{proof}
    The K\H{o}v\'ari--S\'os--Tur\'an theorem implies that any bipartite graph with parts of order $s \le t$ and density $q$ with $t \geq q^{-10}$ has a complete subgraph with $\Omega(\log s/\log (2/ q))$ vertices in the part of order $s$ and at least $t^{0.9}$ vertices in the part of order $t$.

    Let $\chi_{f,g}$ be a pair construction for $K_{N}^{(3)}$ with complexity $p$ and $N=2^{C n \log p}$ for some large constant $C>0$. We would like to show that $\chi$ contains a monochromatic $K_{n,n,n}^{(3)}$. Divide the vertex set into three intervals $I, J, K$ of order $N/3$. We will apply the K\H{o}v\'ari--S\'os--Tur\'an theorem to shrink $I, J, K$ until the edges between $I\times J$, $J \times K$ and $K\times I$ are each monochromatic in $f$.

    First, we build an auxiliary complete bipartite graph $B = (I, J\times K; E)$ and define an edge-coloring of $B$ by $\phi(i, (j,k)) = (f(i,j), f(i,k))$. Observing that $\phi$ takes at most $p^2$ values, there must exist $c_1$ and $c_2$ for which $\phi(i,(j,k)) = (c_1, c_2)$ for at least $|I| |J| |K|/p^2$ choices of $(i,j,k) \in I\times J \times K$. By the K\H{o}v\'ari--S\'os--Tur\'an theorem, we obtain subsets $I'\subseteq I$, $S \subseteq J\times K$ for which $|I'| = \Omega(\log N/\log(2p^2))$, $|S| \ge (|J||K|)^{0.9}$ and $\phi$ is constant on $I'\times S$. Letting $J'$, $K'$ be the set of all $j\in J$, $k \in K$ that appear in $S$, this implies that $|J'|\ge |S|/|K| \ge |J|^{0.8}$, $|K'|\ge |S|/|J| \ge |K|^{0.8}$ and $f(i,j)=c_1$, $f(i,k) = c_2$ for all $(i,j,k)\in I'\times J'\times K'$.

    Let $c_3$ be the most common value of $f(j,k)$ between $J'$ and $K'$. By the K\H{o}v\'ari--S\'os--Tur\'an theorem, we can find subsets $J''\subseteq J'$, $K''\subseteq K'$ where $|J''|, |K''| = \Omega(\log N/ \log (2p))$ and $f(j, k) = c_3$ for all pairs $(j,k) \in J'' \times K''$. Observe that, by our choice of $N$, $|I'|, |J''|, |K''| = \Omega (\log N/ \log (2p^2)) \ge n$. Since every triple $(i,j,k) \in I' \times J'' \times K''$ is assigned the same color $\chi_{f,g}(i,j,k)=g(c_1, c_2, c_3)$, we are done.
\end{proof}

Therefore, if $N = r(K_{n,n,n}^{(3)}) = 2^{\Theta(n^2)}$, the minimum complexity needed to obtain $r^p(K_{n,n,n}^{(3)}) \ge N^{\Omega(1)}$ is $p=2^{\Omega(n)} \ge 2^{\Omega(\sqrt{\log N})}$. One can show that this is tight by picking uniformly random $f$ and $g$.

\subsection{What growth rates are possible?} 

Adding to the earlier results of Fox and He~\cite{FH}, this paper shows that there are many fixed $3$-graphs $H$ for which $r(H, K_{n}^{(3)}) = 2^{\Theta_H(n \log n)}$. Strangely, we do not have any examples of $H$ for which $r(H, K_{n}^{(3)}) = 2^{\Theta_H(n)}$, though it is natural to conjecture their existence. 

\begin{conj}\label{conj:exponential}
There exists a $3$-graph $H$ for which $r(H, K_{n}^{(3)}) = 2^{\Theta_H(n)}$.
\end{conj}

We now explicitly construct a $3$-graph $H$ with $\frac{1}{3}<\mpair(H)<\frac{1}{2}$. Since \cref{thm:mainmpair} does not apply to this $H$ and yet $r(H, K_{n}^{(3)}) \ge 2^{\Omega_H(n)}$ by \cref{thm:EHmpair}, it may be a promising candidate for verifying \cref{conj:exponential}.

\begin{prop}\label{mpairexistence}
    There exists a $3$-graph $H$ with $\frac{1}{3}<\mpair(H)<\frac{1}{2}$.
\end{prop}
\begin{proof}
First we show that if $H$ can be constructed by repeatedly adding edges $\{u,v,w\}$ such that at most one of the three pairs $\{u,v\}$, $\{v,w\}$ and $\{w,u\}$ was previously in an edge together, then $\mpair(H)<1/2$. Indeed, suppose $f:\partial H \rightarrow V(G)$ is a pair homomorphism with $e(G)/v(G) < \frac{1}{2}$, $H' = H\cup \{u,v,w\}$ and, without loss of generality, among the three pairs only $\{u,v\}$ lies in the shadow of $H$. Then there is a pair homomorphism $f':\partial H' \rightarrow V(G')$ where $G'$ is obtained from $G$ by adding two new vertices $x=f'(v,w)$ and $y=f'(w,u)$ and one new edge $(f(u,v), x, y)$. But then 
\[
\frac{e(G')}{v(G')} = \frac{e(G) + 1}{v(G)+2} < \frac{1}{2}.
\]

A link path is the link $3$-graph of a path. We construct $H$ by linearly gluing two link paths, by which we mean that we identify some of the vertices in the two paths while insisting that they do not share edges. 
Such a gluing satisfies the property above, since each link path can be built iteratively, and so $\mpair(H)<\frac{1}{2}$. It remains to check that $\mpair(H) > \frac{1}{3}$, i.e., that $H$ is not $123$-inducible.

Our construction is as follows. Let $U$ with vertices $u,u_1,\ldots,u_{12}$ and $V$ with vertices $v,v_1,\ldots,v_{12}$ be the links of the paths $u_1u_2\cdots u_{12}$ and $v_1v_2\cdots v_{12}$, respectively. We identify certain $u_i$ to certain $v_j$ as follows:
\begin{itemize}
    \item Glue $v_1$ to $u_2$, $v_2$ to $u_5$, $v_5$ to $u_6$ and $v_6$ to $u_1$.

    \item Glue $v_7$ to $u_{11}$, $v_8$ to $u_8$, $v_{11}$ to $u_7$ and $v_{12}$ to $u_{12}$.
\end{itemize}
It is easy to check that this is a linear gluing. We call the resulting graph $H$ and claim that it is not $123$-inducible.

\begin{lemma}
    If $W$ is the $w$-link of the path $w_1w_2\cdots w_k$, then, in any total order on $w,w_1,\ldots,w_k$ that leads to a valid $123$-coloring of $W$, either $w_1<w_2>w_3<w_4>\cdots$ or $w_1>w_2<w_3>w_4<\cdots$.
\end{lemma}

\begin{proof}
    For any valid $123$-coloring and any edge $abc$, we may determine the relative order of $a$, $b$ and $c$ if we know the label of the $2$-edge $ab$ and the relative order of $a$ and $b$. Thus, since the edges $ww_iw_{i+1}$ and $ww_{i+1}w_{i+2}$ share the $2$-edge $ww_{i+1}$, $ww_{i+1}w_i$ and $ww_{i+1}w_{i+2}$ must have the same relative order. In particular, $w_i<w_{i+1}$ if and only if $w_{i+1}>w_{i+2}$, completing the proof of the lemma.
\end{proof}

Using this lemma, it is not difficult to show that $H$ is not $123$-inducible. Suppose, for the sake of contradiction, that we have a valid total ordering on the vertices of $H$ leading to a $123$-coloring. We have four cases:

\medskip

\noindent
Case 1: $u_1<u_2$ and $v_1<v_2$. In this case, the lemma implies $u_5<u_6$ and $v_5<v_6$. However, $v_1=u_2$ and $v_2=u_5$, so $u_1<u_2=v_1 < v_2 = u_5 < u_6$. However, $v_5=u_6$ and $v_6=u_1$, contradicting $v_5<v_6$.

\noindent
Case 2: $u_1>u_2$ and $v_1>v_2$. This is exactly the same as Case 1, but with the inequality signs flipped.

\noindent
Case 3: $u_1<u_2$ and $v_1>v_2$. The lemma implies that $u_7<u_8$ and $u_{11}<u_{12}$, while $v_7>v_8$ and $v_{11}>v_{12}$. Since $v_7=u_{11}$ and $v_8=u_8$, we have that $u_7<u_8=v_8<v_7=u_{11}<u_{12}$. But $v_{11}=u_7$ and $v_{12}=u_{12}$, contradicting $v_{11}>v_{12}$.

\noindent
Case 4: $u_1>u_2$ and $v_1<v_2$. This is exactly the same as Case 3, but with the inequality signs flipped.

\medskip

Since we have a contradiction in each case, $H$ is not $123$-inducible, as required.
\end{proof}


    

If we allow families $\cal H$ instead of single 3-graphs $H$, then the growth rate $2^{\Theta_{\cal H}(n)}$ can occur. 
Indeed, letting $\cal H$ be the family of links of nonbipartite graphs, we quickly see that the 3-graph whose edges are the cyclic triangles in a random tournament provides the lower bound. On the other hand, the upper bound $2^{n}-1$ follows by applying induction to the larger part of the bipartition determined by the link of a vertex in an $\cal H$-free 3-graph. 

More generally, one might ask what growth rates are possible. For instance, are there fixed $3$-graphs $H$ for which $r(H, K_{n}^{(3)})$ equals, say, $2^{\Theta_H(\sqrt{n})}$, $2^{\Theta_H(n \log \log n)}$ or even $2^{\Theta_H(n \log^* n)}$? 

However, the outstanding question on off-diagonal hypergraph Ramsey numbers is to decide whether there are fixed $3$-graphs $H$ for which $r(H, K_{n}^{(3)}) \ge 2^{\omega(n \log n)}$. It seems likely that this should already be the case for $H = K_4^{(3)}$, but our methods seem insufficient for proving this. We even suspect the following should be true.

\begin{conj}
For every $c > 0$, there exists $s$ such that $r(K_s^{(3)}, K_{n}^{(3)}) \ge 2^{\Omega_s(n^c)}$.
\end{conj}

We note that this is not true if we replace $K_n^{(3)}$ by $K_{n,n,n}^{(3)}$, since it is easy to see that $r(H, K_{n,n,n}^{(3)})\le 2^{O_H(n^2)}$ for all $H$. In fact, using the methods of~\cite{CFS2}, one can show that for every $3$-graph $H$ there exists $\epsilon > 0$ such that $r(H, K_{n,n,n}^{(3)}) \le 2^{O_H(n^{2-\epsilon})}$. We suspect that it may even be the case that $r(H, K_{n,n,n}^{(3)}) \le 2^{O_H(n \log n)}$ for all $H$, which would go some way towards explaining the difficulties in bypassing this bound.

\subsection{Linear hypergraphs} 

The conjecture that $r(F, K_n^{(3)})$ is at most polynomial in $n$ for any linear $F$ was a special case of another unpublished conjecture, saying that if $F$ is $123$-inducible, then $r(F, K_n^{(3)})$ is at most polynomial in $n$. If that had been true, then, together with Theorem~\ref{thm:EHmpair}, we would have had a clean characterization saying that $r(F, K_n^{(3)})$ is at most polynomial in $n$ if $F$ is $123$-inducible and at least exponential in $n$ if $F$ is not $123$-inducible. Instead, Theorem~\ref{thm:linearsteppingup} leaves us in a messier, though arguably more interesting, situation. If $F$ is linear with $s$ vertices, then $r(F, K_n^{(3)}) \le r(K_s^{(3)}, K_n^{(3)})$ and the best known upper bound for $r(K_s^{(3)}, K_n^{(3)})$ is
$2^{O(n^{s-2}\log n)}$ from~\cite{CFS}. However, for linear $F$, the arguments in~\cite{CFS} can be extended to prove the stronger upper bound $r(F, K_n^{(3)}) \le 2^{n^{s-\Omega(s^{1/2})}}$. We even believe that the following may be true.

\begin{conj}
There is an absolute constant $C>0$ such that $r(F, K_n^{(3)}) \le 2^{O_F(n^C)}$ for every fixed linear $3$-graph $F$.
\end{conj}

Perhaps this conjecture could also hold with $n^C$ replaced by $n^{o(1)}$ or with the family of linear $3$-graphs replaced by the family of $123$-inducible $3$-graphs, though we are rather less certain about these possibilities.

Finally, we would also like to know whether Theorem~\ref{thm:linearsteppingup} holds not just for some linear $F$ but for most linear $F$. To show this, one would just need to show that the pseudorandomness conditions proved in Lemma~\ref{steiner} hold with high probability for random (partial) Steiner systems. We may return to this problem, whose study would require methods rather different to those used here, in future work.

\vspace{3mm}
\noindent
{\bf Acknowledgements.} This research was initiated during a visit to the American Institute of Mathematics under their SQuaREs program.


\begin{thebibliography}{99}

\bibitem{AKS}
M. Ajtai, J. Koml\'os and E. Szemer\'edi, A note on Ramsey numbers, {\it J. Combin. Theory Ser. A} {\bf 29} (1980), 354--360.
 
 \bibitem{BFM} T. Bohman, A. Frieze and D. Mubayi,
 Coloring $H$-free hypergraphs,  {\it Random Structures Algorithms} {\bf 36} (2010), 11--25.
 

\bibitem{CGMS} M. Campos, S. Griffiths, R. Morris and J. Sahasrabudhe, An exponential improvement for diagonal Ramsey, preprint available at  arXiv:2303.09521 [math.CO].


\bibitem{CFS} D. Conlon, J. Fox and B. Sudakov, Hypergraph Ramsey numbers, {\it J. Amer. Math. Soc.} {\bf 23}
(2010), 247\textendash266.

\bibitem{CFS2} D. Conlon, J. Fox and B. Sudakov, Erd\H{o}s--Hajnal-type theorems in hypergraphs, {\it J. Combin. Theory Ser. B} {\bf 102} (2012), 1142--1154.

\bibitem{CFS3} D. Conlon, J. Fox and B. Sudakov, An improved bound for the stepping-up lemma, {\it Discrete Appl. Math.} {\bf 161} (2013), 1191\textendash1196.


\bibitem{EH} P. Erd\H os and A. Hajnal, On Ramsey like theorems. Problems and results, in Combinatorics (Proc. Conf. Combinatorial Math., Math. Inst., Oxford, 1972), 123--140, Inst. Math. Appl., Southend-on-Sea, 1972.


\bibitem{ESz} P. Erd\H os and G. Szekeres, A combinatorial problem in geometry, {\it Compos. Math.} {\bf 2} (1935), 463--470.

\bibitem{FH} J. Fox and X. He, Independent sets in hypergraphs with a forbidden link, {\it Proc. London Math. Soc.} {\bf 123} (2021), 384--409.

 \bibitem{GRS} R. L. Graham, B. L. Rothschild and J. H. Spencer, {\bf Ramsey Theory}, 2nd Edition, Wiley-Intersci. Ser. Discrete Math. Optim., John Wiley \& Sons, Inc., New York, 1990.

\bibitem{Ki} J. H. Kim, The Ramsey number $R(3, t)$ has order of magnitude $t^2/\log t$, {\it Random Structures Algorithms} {\bf 7} (1995), 173\textendash207.

\bibitem{Kr} M. Krivelevich, Bounding Ramsey numbers through large deviation inequalities, {\it Random Structures  Algorithms} {\bf 7} (1995), 145--155.

\bibitem{MaV} S. Mattheus and J. Verstra\"ete, The asymptotics of  $r(4,t)$, {\it Annals of Math.} {\bf 199} (2024), 919--941.

\bibitem{Mu} D. Mubayi, Improved bounds for the Ramsey number of tight cycles versus cliques, {\it Combin. Probab. Comput.} {\bf 25} (2016), 791--796.




\bibitem{MSBLMS} D. Mubayi and A. Suk, New lower bounds for hypergraph Ramsey numbers, {\it Bull. London Math. Soc.} {\bf 50} (2018), 189--201.
 

\bibitem{Ram} F. P. Ramsey, On a problem of formal logic, {\it Proc. London Math. Soc.} {\bf 30} (1930), 264--286.


\end{thebibliography}
\end{document}